\documentclass[12pt,reno]{amsart}
\usepackage{amssymb}
\usepackage{amsmath}
\usepackage{enumitem}
\usepackage{mathrsfs}
\usepackage[english]{babel}
\usepackage[all]{xy}
\usepackage{hyperref}
\usepackage{marginnote}

\usepackage{stmaryrd}

\usepackage[margin=1.15in]{geometry}

\RequirePackage{mathrsfs} %\let\mathcal\mathcal

\newtheorem{theorem}{Theorem}[section]
\newtheorem{lemma}[theorem]{Lemma}
\newtheorem{proposition}[theorem]{Proposition}
\newtheorem{corollary}[theorem]{Corollary}
\newtheorem{conjecture}[theorem]{Conjecture}

\theoremstyle{definition}
\newtheorem*{ack}{Acknowledgements}
\newtheorem*{con}{Conventions}
\newtheorem{remark}[theorem]{Remark}

\newtheorem{definition}[theorem]{Definition}

\numberwithin{equation}{section} \numberwithin{figure}{section}

 \DeclareMathOperator{\NS}{NS}
\DeclareMathOperator{\Aut}{Aut} 
\DeclareMathOperator{\Bir}{Bir}

\DeclareMathOperator{\Spec}{Spec}

\DeclareMathOperator{\an}{an}

\DeclareMathOperator{\Hom}{Hom}

\newcommand{\Qbar}{\overline{\QQ}}

\newcommand\ZZ{\mathbb{Z}}

\newcommand\QQ{\mathbb{Q}}
\newcommand\RR{\mathbb{R}}
\newcommand\CC{\mathbb{C}}

\usepackage{color}

\title[Finiteness properties of pseudo-hyperbolic varieties]{Finiteness properties of pseudo-hyperbolic varieties}

\author{Ariyan Javanpeykar}
\address{Ariyan Javanpeykar \\
Institut f\"{u}r Mathematik\\
Johannes Gutenberg-Universit\"{a}t Mainz\\
Staudingerweg 9, 55099 Mainz\\
Germany.}
\email{peykar@uni-mainz.de}

\author{Junyi Xie}
\address{Junyi Xie \\
IRMAR \\
Campus de Beaulieu \\
b\^atiments 22 et 23 \\
263 avenue du G\'en\'eral Leclerc, CS 74205 \\
35042  Rennes C\'edex \\
France.}
\email{junyi.xie@univ-rennes1.fr}

 %hyperbolicity


\begin{document}

\begin{abstract}
Motivated by Lang-Vojta's conjecture, we show that the  set of   dominant  rational self-maps of an algebraic variety over a number field with only finitely many rational points in any given number field    is finite by combining Amerik's theorem for dynamical systems of infinite order with  properties of  Prokhorov-Shramov's notion of quasi-minimal models. We also prove a similar result in the geometric setting by using again  Amerik's theorem and Prokhorov-Shramov's notion of quasi-minimal model, but also  Weil's regularization theorem for birational self-maps and properties of   dynamical degrees. Furthermore, in the geometric setting, we   obtain an analogue of Kobayashi-Ochiai's finiteness result for varieties of general type, and thereby generalize Noguchi's theorem (formerly Lang's conjecture).  
% Motivated by Lang-Vojta's conjecture, we    show that the   set of   dominant  rational self-maps of an algebraic variety over a number field with only finitely many rational points in any given number field    is finite by combining Amerik's theorem for dynamical systems of infinite order with  properties of  Prokhorov-Shramov's notion of quasi-minimal models. We also prove a similar result in the geometric setting by using again  Amerik's theorem and Prokhorov-Shramov's notion of quasi-minimal model, but also  Weil's regularization theorem for birational self-maps and properties of    dynamical degrees. Furthermore, in the geometric setting, we   obtain an analogue of Kobayashi-Ochiai's finiteness result for varieties of general type, and thereby generalize Noguchi's theorem (formerly Lang's conjecture).    Our proof here relies on  a deformation-theoretic result for surjective maps of normal   varieties due to Hwang-Kebekus-Peternell. 
\end{abstract}

\maketitle
 \tableofcontents

\thispagestyle{empty}

\section{Introduction}   
 Lang-Vojta's conjectures predict that  a variety  over the field of rational numbers with only finitely many rational points over any given number field is of general type, and should therefore have only finitely many birational automorphisms. In this paper we verify this prediction by building on the first-named author's earlier work \cite{JAut}. We also use the results of the first-named author with van Bommel and Kamenova \cite{vBJK, JKa} to obtain similar and stronger results in the ``geometric'' setting.
 
 More precisely, we    show that the   set of   dominant  rational self-maps of an algebraic variety over a number field with only finitely many rational points in any given number field    is finite by combining Amerik's theorem for dynamical systems of infinite order with  properties of  Prokhorov-Shramov's notion of quasi-minimal models. We also prove a similar result in the geometric setting by using again  Amerik's theorem and Prokhorov-Shramov's notion of quasi-minimal model, but also  Weil's regularization theorem for birational self-maps and properties of    dynamical degrees. Furthermore, in the geometric setting, we   obtain an analogue of Kobayashi-Ochiai's finiteness result for varieties of general type, and thereby generalize Noguchi's theorem (formerly Lang's conjecture).    Our proof here relies on  a deformation-theoretic result for surjective maps of normal   varieties due to Hwang-Kebekus-Peternell.

\subsection{The arithmetic Lang-Vojta conjecture} Throughout this paper, we let $k$ be an algebraically closed field of characteristic zero.  A variety over $k$ is a finite type separated scheme over $k$. If $X$ is a variety over $k$ and   $A\subset k$ is a subring, then a \emph{model for $X$ over $A$} is a pair $(\mathcal{X},\phi)$ with $\mathcal{X}\to \Spec A$ a finite type separated scheme and $\phi:\mathcal{X}\otimes_A k \to X$ an isomorphism of schemes over $k$. We will often omit $\phi$  from our notation.

 A proper variety $X$ over $k$ is \emph{pseudo-Mordellic  over $k$} if there is a proper closed subset $\Delta\subsetneq X$ such that, for every finitely generated subfield $K\subset k$ and every model $\mathcal{X}$ for $X$ over $K$,  the set   $\mathcal{X}(K)\setminus \Delta$   is finite. In words, a proper variety over $k$ is pseudo-Mordellic if it has only finitely many rational points outside some ``exceptional'' subset. The ``correct'' definition of pseudo-Mordellicity for non-proper varieties is slightly more involved, and was   first given by Vojta \cite{VojtaLangExc} (see also Section \ref{section:mordell}).

The property of being pseudo-Mordellic for a proper variety $X$ is conjecturally related to the positivity of the canonical bundle  on some desingularization of $X$. The precise conjecture is due to Lang and Vojta. To state this conjecture, recall that a proper   variety $X$ over $k$ is of general type if, for every irreducible component $Y$ of $X$ and every resolution of singularities $\widetilde{Y}\to Y$, the canonical bundle $\omega_{\widetilde{Y}}$ of $\widetilde{Y}$ is big (see \cite{Lazzie1,Lazzie2}). For example, a smooth projective connected curve over $k$ is of general type if and only if it has genus at least two.

\begin{conjecture}[Lang-Vojta, I]\label{conj:lang}
A proper integral variety $X$ over $k$   is of general type if and only if $X$ is pseudo-Mordellic over $k$.
\end{conjecture}

 This conjecture  first appeared in   \cite{Lang2}; see also   \cite[Conjecture~XV.4.3]{CornellSilverman},  \cite[\S0.3]{Abr} and \cite{JBook}. Faltings's theorem (\emph{formerly} Mordell's conjecture)  shows that this conjecture holds for one-dimensional proper varieties \cite{Faltings2, FaltingsComplements}. More generally, 
 Conjecture \ref{conj:lang} is known to hold for closed subvarieties of abelian varieties by work of Faltings, Kawamata and   Ueno  \cite{FaltingsLang, Kawamata, Ueno}.    
 
  The above version of Lang-Vojta's conjecture predicts arithmetic finiteness properties for varieties of general type, and is only a ``small'' part of the more general conjecture. We refer the reader to \cite[\S12]{JBook} (and especially \cite[Conjecture~12.1]{JBook}) for a   complete statement of Lang-Vojta's conjectures for projective varieties. 
  
 In \cite[Conj.~4.3]{Vojta3} Vojta extended this conjecture to quasi-projective varieties (see also \cite{VojtaLangExc}). Furthermore, in \cite{Campana, CampanaOr} Campana expanded upon Lang-Vojta's conjectures and conjectured that a smooth projective connected variety $X$ over $k$ is special (as defined in \emph{loc. cit.}) if and only if there is a finitely generated subfield $K\subset k$ and a model $\mathcal{X}$ for $X$ over $k$ such that $\mathcal{X}(K)$ is dense in $X$. 
 
 The  first striking consequence of  Lang-Vojta's conjecture was obtained by Caporaso-Harris-Mazur \cite{CHM} (see  \cite{AscherTurchet} for generalizations); we refer the reader to \cite{JL, JLFano} for other consequences of Lang-Vojta's conjecture.

 \subsection{Finiteness of   dominant rational self-maps: arithmetic setting}
 It is well-known that a  curve of genus at least two has only finitely many automorphisms.  A vast generalization of this finiteness statement for higher-dimensional varieties is provided by Matsumura's finiteness theorem for varieties of general type \cite[\S 11]{Iitaka}.  
 
 \begin{theorem}[Matsumura] \label{thm:mat} If $X$ is a proper integral variety of general type over $k$, then the set of dominant rational self-maps $X\dashrightarrow X$ over $k$ is finite.
 \end{theorem}
 
In light of Matsumura's finiteness result (Theorem \ref{thm:mat}) and Lang-Vojta's conjecture (Conjecture \ref{conj:lang}), one expects the finiteness of the set of dominant rational self-maps of a proper pseudo-Mordellic variety. Our first result verifies this expectation, and reads as follows.
 
 \begin{theorem}[Main Result, I] \label{thm:birs} 
 If $X$ is a proper pseudo-Mordellic integral variety over $k$, then  the set of dominant rational self-maps  $X\dashrightarrow X$ over $k$ is finite.
 \end{theorem}
 
 A more concrete way of stating Theorem \ref{thm:birs} is as follows. Let $X$ be a variety over a finitely generated field $K$ of characteristic zero, and let $K\to \overline{K}$ be an algebraic closure of $K$. Suppose that, for every finite field extension $L$ of $K$, the set of $L$-rational points on $X$ is finite. Then, the set of   dominant rational self-maps $X_{\overline{K}}\dashrightarrow X_{\overline{K}}$ of $X_{\overline{K}}$ over   $\overline{K}$ is finite. Indeed, the finiteness hypothesis on the sets of $L$-rational points on $X$ implies that any  compactification of $X_{\overline{K}}$ is pseudo-Mordellic over  $\overline{K}$.
 
 If $X$ is Mordellic, then it is shown in   \cite{JAut} that $\Aut(X)$ is finite. One novel aspect of  this paper is the use of quasi-minimal models, as introduced by Prokhorov-Shramov  \cite{ProShramov2014}.  This will allow us to leverage some of the techniques used to handle automorphisms in \emph{loc. cit.} to   handle birational self-maps.

\subsection{Finiteness of dominant rational self-maps: geometric setting}
The arithmetic Lang-Vojta conjecture (Conjecture \ref{conj:lang}) relates varieties of general type to varieties with only finitely many rational points. However, 
 another part of Lang-Vojta's conjectures   relates varieties of general type to varieties satisfying certain boundedness properties (see Conjecture \ref{conj:lang20} below). Such boundedness properties   are studied systematically in \cite{vBJK, Demailly, JKa}, and come in many \emph{a priori different} guises (see Sections \ref{section:alg} and \ref{section:geomhyp}). 
 
 We follow \cite{vBJK, JKa} and say that a projective variety $X$  over $k$ is \emph{pseudo-$1$-bounded} over $k$ if there
   is a proper closed subset $\Delta\subsetneq X$ of $X$ such that, for every ample line bundle $L$ on $X$ and every smooth projective connected curve $C$ over $k$, there is a real number $\alpha(X,\Delta, L, C)$ depending only on $X, \Delta, L$, and $C$, such that, for every morphism $f:C\to X$, the inequality 
 \[
 \deg_C f^\ast L\leq \alpha(X,\Delta,L,C)
 \] holds. In words,  a projective variety over $k$ is pseudo-$1$-bounded if the moduli spaces of curves mapping to $X$ and not landing in some ``exceptional'' subset are of finite type, i.e.,  the ``height'' of  a curve in $X$ outside some ``exceptional'' subset is   bounded. 
 
 The following version of Lang-Vojta's conjecture predicts that projective varieties of general type are pseudo-$1$-bounded.
 
  \begin{conjecture}[Lang-Vojta, II]\label{conj:lang20}
 A  projective integral variety $X$ over $k$ is of general type if and only if $X$ is pseudo-$1$-bounded over $k$.
 \end{conjecture}

 This ``geometric'' version of Lang-Vojta's conjecture is known for closed subvarieties of abelian varieties and, more generally, varieties with maximal Albanese dimension (see \cite{Yamanoi1}). Moreover, this conjecture holds   for surfaces   satisfying $c_1^2 > c_2$  by work of Bogomolov and McQuillan \cite{DeschampsBogomolov, McQuillan}.

In light of Matsumura's theorem (Theorem \ref{thm:mat}) and the second version of Lang-Vojta's conjecture (Conjecture \ref{conj:lang20}), one expects  that  a pseudo-$1$-bounded projective variety $X$ admits only finitely many dominant rational self-maps. Our second result verifies this expectation, and reads as follows.
 
\begin{theorem}[Main Result, II]\label{thm:birs2}
If $X$ is a pseudo-$1$-bounded   integral projective variety  over $k$, then the set of dominant rational self-maps $X\dashrightarrow X$ is finite.
\end{theorem}

 \subsection{Ingredients of the proofs}
 In our proofs of Theorems \ref{thm:birs} and \ref{thm:birs2},  we will use a criterion for the finiteness of the group $\mathrm{Bir}_k(X)$ of birational self-maps $X\dashrightarrow X$ of a proper non-uniruled variety $X$ over $k$. Namely, we will use that, for a proper non-uniruled integral variety $X$ over $k$,  the group $\mathrm{Bir}_k(X)$ is finite if and only if it is torsion. 
 
 \begin{theorem}[Main Result III]  \label{thm:tor}   Let $X$ be a proper integral non-uniruled  variety over $k$. If $\Bir_k(X)$ is torsion, then $\Bir_k(X)$ is finite. 
 \end{theorem} 

The proof of this finiteness criterion (Theorem \ref{thm:tor}) crucially relies on  Prokhorov-Shramov's aforementioned theory of quasi-minimal models. In fact, we use their theory to show that Theorem \ref{thm:tor} follows from    the fact that the automorphism group of a projective variety over $k$ is finite if and only if it is torsion (see \cite{JAut}).

Our proofs of Theorem \ref{thm:birs} and \ref{thm:birs2} both use Theorem \ref{thm:tor}. Another common ingredient of both proofs is Amerik's theorem on algebraic dynamics. To state Amerik's theorem, for $f:X\dashrightarrow X$ a dominant rational self-map of a variety $X$ over $k$, we let $\mathrm{In}(f)$ denote its locus of indeterminacy (see Section \ref{section:ends} for a detailed discussion of Amerik's theorem).

 \begin{theorem} [Amerik]\label{thm:Amerik} Let $k$ be  an algebraically closed field of characteristic zero, let $X$ be a variety over $k$, and let $f:X\dashrightarrow X$ be a  dominant rational self-map of infinite order. Then, for every non-empty open subset $U\subset X$, there is a point $x$ in $U\setminus \mathrm{In}(f)$ such that the orbit of $x$ is infinite and is contained in  $U\setminus \mathrm{In}(f)$.
\end{theorem}

Amerik's theorem   allows us to show that every dominant rational self-map of a proper pseudo-Mordellic   variety is in fact a birational automorphism of finite order, so that the ``arithmetic'' finiteness result (Theorem \ref{thm:birs}) follows immediately from Theorem \ref{thm:tor}. 
On the other hand, 
the application of Amerik's theorem in  our proof of the ``geometric'' finiteness result (Theorem \ref{thm:birs2}) is not as immediate. Indeed, the proof of the ``geometric'' finiteness result uses, in addition to Amerik's theorem, also   Weil's regularization theorem (see Theorem \ref{thm:weil}) and the notion of dynamical degrees.

 Interestingly, the ``(complex-analytic) Brody'' analogue of the above   results is currently not known (see \cite[Remark~15.7]{JBook} for a related discussion).    More precisely, define a variety $X$ over $\mathbb{C}$ to be \emph{pseudo-Brody hyperbolic} if there is a proper closed subset $\Delta\subsetneq X$ such that every non-constant holomorphic map $\CC\to X^{\an}$ factors over $\Delta$. Then, the Green-Griffiths-Lang conjecture says that a proper variety  $X$ over $\CC$ is of general type if and only if $X$ is pseudo-Brody hyperbolic. Note that this conjecture  predicts the finiteness of the group $\Aut(X)$ of automorphisms of $X$ when $X$ is a proper pseudo-Brody hyperbolic variety over $\CC$. This finiteness is however not known to hold for all pseudo-Brody hyperbolic proper varieties over $\CC$ (of dimension at least three).

\subsection{The moduli of surjective morphisms} Theorem  \ref{thm:birs} and Theorem \ref{thm:birs2} verify the finiteness of the group of birational self-maps for a pseudo-Mordellic  and pseudo-$1$-bounded projective variety, respectively. However,  it is natural to expect  \emph{stronger} finiteness results.  Indeed, a classical result of de Franchis-Severi says that, given two curves $X$ and $Y$ of genus  at least two, the set of dominant morphisms   from $Y$ to $ X$ is finite.   
Kobayashi and Ochiai proved a vast generalization of this finiteness statement for higher-dimensional varieties of general type; see \cite[\S7]{KobayashiBook}.

\begin{theorem}[Kobayashi-Ochiai]
If $X$ is a projective     variety over $k$ of general type and $Y$ is a   projective integral variety   over $k$, then the set of surjective morphisms    $f\colon Y\to X$ is finite.
\end{theorem}

In light of Lang-Vojta's conjectures and Kobayashi-Ochiai's theorem, any ``pseudo-hyperbolic'' variety should satisfy a similar finiteness property. 

First, let us note that the analogue of Kobayashi-Ochiai's theorem for pseudo-Mordellic projective varieties is not known. That is, given a projective pseudo-Mordellic variety $X$ over $k$ and a projective integral variety $Y$ over $k$, it is not known whether the set of surjective morphisms $Y\to X$ is finite. In this paper  we ``only'' show that every  surjective morphisms $Y\to X$ is rigid. i.e., the scheme parametrizing morphisms $Y\to X$ is zero-dimensional; see Corollary \ref{cor:sur_is_rigid} for a precise statement. We  thereby provide a first step towards proving this expected finiteness property for pseudo-Mordellic proper varieties.

However, in the case of pseudo-$1$-bounded  projective varieties, we   prove   the expected analogous finiteness statement by combining our earlier finiteness result (Theorem \ref{thm:birs2}) with a deformation-theoretic result of Hwang-Kebekus-Peternell \cite{HKP}.
 The precise statement of our result is as follows. 

 \begin{theorem}[Main Result, IV]\label{thm:sur_intro1} Assume $k$ is \textbf{uncountable}.
If $X$ is a projective   pseudo-$1$-bounded   variety over $k$ and $Y$ is a   projective integral variety   over $k$, then the set of surjective morphisms    $f\colon Y\to X$ is finite.
 \end{theorem}

 Since Brody hyperbolic projective varieties over $\CC$ are $1$-bounded, this theorem generalizes Horst-Noguchi's theorem (formerly Lang's conjecture) and also the generalization of Horst-Noguchi proven in \cite{JKa}. We refer the reader to Remark \ref{remark:noguchi} for a more detailed discussion.
 
\subsection{Pseudo-algebraic hyperbolicity} The uncountability hypothesis  in Theorem \ref{thm:sur_intro1} can be removed under a slightly stronger (but conjecturally equivalent) boundedness condition.  Let $X$ be a projective variety over $k$ and let $\Delta\subset X$ be a closed subset. Following Demailly \cite{Demailly} (see also \cite{vBJK, JKa,  PaRou, ARW, Rou1, Rou2}), we  say that $X$ is \emph{algebraically hyperbolic over $k$ modulo $\Delta$} if, for every ample line bundle $L$ on $X$, there is a real number $\alpha(X,\Delta,L)$ depending only on $X$, $\Delta$ and $L$ such that, for every smooth projective connected curve $C$ over $k$ and every morphism $f:C\to X$ with $f(C)\not \subset \Delta$, the inequality
 \[
 \deg_C f^\ast L \leq \alpha(X,\Delta, L) \cdot \mathrm{genus}(C)
 \] holds.
Furthermore,
 a projective variety  $X$ over $k$ is \emph{pseudo-algebraically hyperbolic (over $k$)} if there is a proper closed subset $\Delta\subsetneq X$ such that $X$ is algebraically hyperbolic modulo $\Delta$.   
 
The Lang-Vojta conjectures predict that a projective variety of general type over $k$ is not only pseudo-$1$-bounded, but even pseudo-algebraically hyperbolic. That is, roughly speaking, one does not only expect the boundedness of the degrees of  morphisms $f:C\to X$ with $f(C)\not\subset \Delta$ (see Conjecture \ref{conj:lang20}), but even more that the degree of such a morphism is bounded linearly in the genus of $C$.
 
 \begin{conjecture}[Lang-Vojta, III]\label{conj:lang2}
 A projective integral variety $X$ over $k$ is of general type if and only if $X$ is pseudo-algebraically hyperbolic over $k$.
 \end{conjecture}

 Our next result establishes precisely the analogue of Kobayashi-Ochiai's finiteness result for pseudo-algebraically hyperbolic projective varieties, as   predicted by Conjecture \ref{conj:lang2}.
 
 \begin{theorem}[Main Result, V]\label{thm:sur_intro}
If $X$ is a projective   pseudo-algebraically hyperbolic  variety over $k$ and $Y$ is a   projective integral variety   over $k$, then the set of surjective morphisms    $f\colon Y\to X$ is finite.
 \end{theorem}

 The above results are motivated by Lang-Vojta's conjecture, and verify predictions made by this conjecture. This is to be contrasted with the  following result for pseudo-algebraically hyperbolic varieties, as this result can not be deduced from the Lang-Vojta conjectures and known properties of varieties of general type. To state this result, recall that for $X$ and $Y$ projective varieties over $k$, we let $\underline{\Hom}_k(Y,X)$ denote the locally finite type scheme parametrizing morphisms from $Y$ to $X$.
 \begin{theorem}[Main Result, VI]\label{thm:final}
  Let $X$ be  a projective variety over $k$ and let $\Delta\subsetneq X$ be a proper closed subset such that $X$ is algebraically hyperbolic modulo $\Delta$. Then the following statements hold.
  \begin{enumerate}
  \item For a projective integral scheme $Y$, the scheme $$\underline{\mathrm{Hom}}_k(Y,X)\setminus \underline{\Hom}_k(Y,\Delta)$$ is quasi-projective over $k$. 
  \item For a projective integral scheme $Y$ over $k$, a non-empty integral closed subscheme $B\subset Y$, and a non-empty integral closed subscheme $A\subset X$ not contained in $\Delta$, the set 
  \[ 
  \Hom_k((Y,B),(X,A)) := \{f:Y\to X \ | \ f(B) = A\}
  \] is finite.
  \end{enumerate}
 \end{theorem}

 This theorem  generalizes certain finiteness results  for Brody hyperbolic projective varieties proven by Kaup, Kiernan, Kobayashi, Urata, and many others; see \cite{KiernanKobayashi, KobayashiBook, Urata}.

 \subsection{Overview} We finish this introduction with a brief overview of some of our related work on Lang-Vojta's conjectures; see \cite{vBJK, JBook, JKuch, JLevin, JSZ, JAut, JKa, JLitt,   JVez}.
 
 First, the conjectures of Lang-Vojta for projective varieties first appeared in Lang's paper \cite{Lang2}. They were expanded upon in the survey paper \cite{JBook}; a complete version of the Lang-Vojta conjectures for projective varieties is stated in \cite[\S 12]{JBook}.  
 
In Section \ref{section:groupless} we study the notion of ``pseudo-groupless'' variety. The term ``groupless'' was first introduced in \cite{JKa, JVez}, and pseudo-groupless varieties are those which are ``groupless modulo a closed subset''. The notion of a groupless variety  is closely related to Lang's notion of ``algebraically(-Lang) hyperbolic'' variety (which is not to be confused with Demailly's notion of algebraic hyperbolicity). The main result of Section \ref{section:groupless} is  that moduli spaces of surjective maps to a pseudo-groupless variety are rigid. This  result has not appeared in the literature before, although a weaker version was proven in \cite{JKa}. A survey of our work on groupless varieties can be found in \cite[\S4 and \S6]{JBook}. 

In this paper we do not study how grouplessness (or any related notion of ``hyperbolicity'') varies in families; the reader interested in such questions is referred to \cite{vBJK, JVez}.   A brief survey of results on hyperbolicity in families is provided in \cite[\S18]{JBook}. We note that the Zariski-countable  openness of the locus of groupless varieties in a projective family of varieties   is proven in \cite{JVez} using non-archimedean extensions of Lang's conjectures.

The main results of this paper on birational automorphisms of pseudo-hyperbolic varieties (Theorems \ref{thm:birs} and \ref{thm:birs2}) build  on our earlier work on automorphisms \cite{JAut} of Mordellic varieties, as well as on our earlier work with Kamenova on automorphisms of algebraically hyperbolic varieties \cite{JKa}. They verify the expected finiteness results for pseudo-Mordellic and pseudo-algebraically hyperbolic varieties; see \cite[\S15]{JBook} for a survey of these results. However,  
 our results however still leave some open questions, and  we refer the reader to \cite[Remark~15.7]{JBook} for a collection of open problems on the dynamics of pseudo-hyperbolic varieties.

Another aspect of our work on Lang-Vojta's conjectures is related to the Persistence conjecture \cite[Conjecture~1.20]{vBJK}: if $L/k$ is an extension of algebraically closed fields of characteristic zero and  $X$ is a Mordellic variety over $k$, then  $X_L$ is Mordellic. This conjecture is resolved in several  cases   in \cite{vBJK,   JLevin, JSZ, JAut,  JLitt}; see \cite[\S17]{JBook} for a survey. Although the Persistence Conjecture for Mordellicity does not play an explicit role in this paper, it seems interesting  to note that if the Persistence Conjecture  were true, then one could have avoided Amerik's theorem in our proof of Theorem \ref{thm:birs}.

\begin{ack}  
The first-named author   acknowledges support from SFB/Transregio 45.
\end{ack}

   \begin{con} Throughout this paper $k$ will be an algebraically closed field of characteristic zero.
 A \emph{variety over  $k$} is a finite type separated     $k$-scheme.  A scheme is \emph{integral} if it is irreducible and reduced.
 \end{con}

   \section{On the finiteness of the group of birational self-maps} 
   In this section we prove that the group of birational automorphisms of a projective non-uniruled variety is finite if and only if it is torsion (Theorem \ref{thm:tor}). Our proof uses  a   finiteness criterion for the automorphism group of a projective variety proven in \cite{JAut} and properties of  quasi-minimal  models of projective non-uniruled varieties.  Our approach is inspired by the work of Jinsong Xu \cite{JinsongXu}.

 If $X$ is a projective variety over $k$,  we let $\mathrm{Bir}_k(X)$ be the group of birational self-maps $X\dashrightarrow X$. Thus, if $X$ is birational to a projective variety $Y$ over $k$, then $\mathrm{Bir}_k(X) = \mathrm{Bir}_k(Y)$.
 Moreover, we let $\mathrm{Aut}_1(X)\subset \mathrm{Bir}_k(X)$  be the subgroup of birational self-maps $X\dashrightarrow X$ over $k$ which are regular in codimension one. That is, $\Aut_1(X)$ is the group of  birational self-maps $X\dashrightarrow X$ which are an automorphism outside a subset of codimension at least two. We refer the reader to \cite{Jelonek}, \cite[\S4]{ProShramov2014} and \cite{JinsongXu} for a discussion of  such birational self-maps.
  
  For $X$ a projective variety over $k$,  recall that $\NS(X)$ denotes the N\'eron-Severi group of $X$.   If $L$ is a line bundle on $X$, we let $[L]$ denote the class of $L$ in $\NS(X)_{\QQ} = \NS(X) \otimes_{\ZZ} \QQ$.
The following  lemma  is well-known (see for  instance   \cite[Lemma~2.4]{JinsongXu}).   We include a short proof for the reader's convenience. 

\begin{lemma}\label{lem:pseudo_is_aut} Let $X$ be a projective variety over $k$ and let $\sigma \in \Aut_1(X)$.  If there is an ample line bundle $L$ on $X$ such that $\sigma^\ast [L] = [L]$ in $\mathrm{NS}(X)_{\mathbb{Q}}$, then $\sigma$ is an automorphism.
\end{lemma}
\begin{proof} 
Let $L$ be an ample line bundle such that $\sigma^\ast [L] = [L]$ in $\NS(X)_{\mathbb{Q}}$.  Let $M_2$ be a $\mathbb{Q}$-divisor in the same class of $L$. Set $M_1:=\sigma^\ast M_2.$ Then  $M_1$ and $M_2$ are numerically equivalent.

 Let $\Gamma$ be the closure of the graph of $\sigma$ in $X\times X$. Let $\pi_1:\Gamma\to X$ be the first projection, and let $\pi_2:\Gamma\to X$ be the second projection. Consider the divisor $D= \pi_1^\ast M_1 - \pi_2^\ast M_2$. Note that $\pi_{1,\ast} D = \pi_{2,\ast} D = 0$.  Since $D$ is $\pi_2$-nef, it follows from the negativity lemma (see \cite[Lemma~3.39]{KollarMori}) that  $-D$ is effective. Similarly, since $-D$ is $\pi_1$-nef, it follows from the negativity lemma that $D$ is effective. Therefore, $D =0$, so that \[
  \pi_1^\ast M_1 = \pi_2^\ast M_2.\]
  In particular, since $\pi_1^\ast L + \pi_2^\ast L$ is ample, we see that  $2\pi_1^\ast L = \pi_1^\ast L + \pi_2^\ast L $ is ample. This implies that $\pi_1^\ast L$ is ample. We conclude that $\pi_1$ is an isomorphism, so that $\sigma$ is an automorphism, as required.
  \end{proof}

  \begin{lemma}\label{lem:aut1} Let $X$ be a projective variety over $k$. If $\mathrm{Aut}_1(X)$ is torsion, then $\mathrm{Aut}(X)\subset \mathrm{Aut}_1(X)$ is of finite index.
  \end{lemma}
  \begin{proof} We adapt the arguments in the proof of \cite{JAut}. 
  Let $\Gamma$ be the kernel of $\Aut_1(X)\to \mathrm{Aut}(\NS(X))$. Since $\Aut_1(X)$ is torsion and $\NS(X)$ is a finitely generated abelian group, the subgroup $\Gamma\subset \Aut_1(X)$  is of finite index  \cite{JAut}. Let $L$ be an ample line bundle, and note that $\Gamma$ fixes the class of $L$ in $\mathrm{NS}(X)$. Therefore, by Lemma \ref{lem:pseudo_is_aut},  every $g$ in $\Gamma$ is an automorphism, so that  $\Gamma\subset \Aut(X)$. Since $\Gamma$ is of finite index in $\Aut_1(X)$ and $\Gamma\subset \Aut(X)$, we conclude  that $\Aut(X)\subset \mathrm{Aut}_1(X)$ is of finite index. 
  \end{proof}
  
  \begin{corollary}\label{cor:aut1_tor_is_fin} Let $X$ be a projective variety over $k$.
  If $\Aut_1(X)$ is torsion, then $\Aut_1(X)$ is finite.
  \end{corollary}
  \begin{proof}
  Assume that  $\Aut_1(X)$ is torsion, so that  $\Aut(X)$ is torsion. Then, it follows from   \cite{JAut} that $\Aut(X)$ is finite. Since $\Aut_1(X)$ is torsion, we also have that $\Aut(X)\subset \Aut_1(X)$ is of finite index (Lemma \ref{lem:aut1}). We conclude that $\Aut_1(X)$ is finite.
  \end{proof}

  \begin{theorem}[Prokhorov-Shramov]\label{thm:qcmodel} Let $X$ be a proper non-uniruled integral variety over $k$. Then  there is a projective integral variety $Y$ and a birational map $X\dashrightarrow Y$  such that $\Bir(X) = \Aut_1(Y)$. 
  \end{theorem}
  \begin{proof} Let $Y$ be a quasi-minimal model for $X$ over $k$; see \cite[\S4]{ProShramov2014}. The existence of such a model follows from \cite[Lemma~4.4]{ProShramov2014}. By  \cite[Corollary~4.7]{ProShramov2014}, we have that  $\Bir(Y) = \Aut_1(Y)$. Since $X$ and $Y$ are birational, it follows that $\mathrm{Bir}(X) = \mathrm{Bir}(Y)$. This shows that $\mathrm{Bir}(X) = \Aut_1(Y)$, as required. 
  \end{proof}

  \begin{proof}[Proof of Theorem \ref{thm:tor}] Since $X$ is non-uniruled, by Theorem \ref{thm:qcmodel}, there is a projective integral variety $Y$ over $k$ and a birational map $X\dashrightarrow Y$ such that $\Bir_k(X) = \Aut_1(Y)$. Therefore, since $\Bir_k(X)$ is torsion, the group $\Aut_1(Y)$ is torsion. Now, as $ \Aut_1(Y)$ is torsion, it follows from Corollary \ref{cor:aut1_tor_is_fin} that   $\Aut_1(Y)= \Bir_k(X)$ is finite.
  \end{proof}

 \section{Pseudo-grouplessness}\label{section:groupless}
 
Roughly speaking, an algebraic variety is groupless if it admits no non-trivial morphisms from an algebraic group; see \cite{JAut, JKa, JVez}. In this paper we introduce the more general notion of pseudo-grouplessness; this property also appears in the work of Lang \cite{Lang2} and Vojta      \cite{VojtaLangExc}.

 \begin{definition}[Pseudo-grouplessness] Let $X$ be a variety over $k$ and let $\Delta\subset X$ be a closed subscheme.
 We say that  $X$  is \emph{groupless modulo $\Delta$ (over $k$)} if,  for every  finite type connected group scheme $G$ over $k$ and every dense open subscheme $U\subset G$ with $\mathrm{codim}(G\setminus U)\geq 2$, every non-constant morphism $U\to X$ factors over $\Delta$. 
\end{definition}

Note that, as $k$ is of characteristic zero, a finite type connected group scheme over $k$ is smooth and quasi-projective over $k$ (see for instance \cite{ConradChevalley}).

\begin{definition}
 A variety $X$ is \emph{pseudo-groupless (over $k$)} if there is a proper closed subset $\Delta\subsetneq X$ such that $X$ is groupless modulo $\Delta$.
\end{definition}

 \begin{remark}\label{remark:properness}
 Let $X$ be a proper variety over $k$ and let $\Delta\subset X$ be a closed subscheme. It follows from the valuative criterion of properness that  $X$ is groupless modulo $\Delta$ if and only if, for every finite type connected group scheme $G$ over $k$ and every dense open subscheme $U\subset G$, every non-constant morphism $U\to X$ factors over $\Delta$.
 \end{remark}

  \begin{remark}
  Let $X$ be a proper variety over $k$ with no rational curves and let $\Delta\subset X$ be a closed subscheme. Then $X$ is groupless modulo $\Delta$ if and only if, for every (smooth) finite type connected group scheme $G$ over $k$, every non-constant morphism $G\to X$ factors over $\Delta$. Indeed, for every dense open $U\subset G$,  as $G$ is smooth over $k$, every morphism $U\to X$ extends to a morphism $G\to X$; see for instance \cite[Corollary~1.44]{Debarrebook1}.
  \end{remark}

\begin{remark}
 If $X$ is an affine variety over $k$, then $X$ is groupless modulo $\Delta$ if and only if, for every finite type connected group scheme $G$ over $k$, every non-constant morphism $G\to X$ factors over $\Delta$. Indeed, if $U\subset G$ is a dense open subset with $\mathrm{codim}(G\setminus U)\geq 2$, it follows from Hartog's lemma and the  smoothness of $G$ over $k$ that  every morphism $U\to X$ extends (uniquely) to a morphism $G\to X$.
\end{remark}

 The notion of ``grouplessness''  (as defined in \cite[Definition~2.1]{JKa}) is well-studied, and sometimes referred to as ``algebraic hyperbolicity'' or ``algebraic Lang hyperbolicity''; see \cite{ShangZhang}, \cite[Remark~3.2.24]{KobayashiBook}, or \cite[Definition~3.4]{KovacsSubs}. As in \cite{vBJK, JKa}, we follow Demailly \cite{Demailly} and reserve ``algebraic hypebolicity'' for his notion of hyperbolicity. 
  
 \begin{remark}\label{rem:gr_is_nonuni}
 A projective pseudo-groupless variety could have rational curves. For example, the blow-up of the product of two smooth proper curves of genus two in a point is pseudo-groupless and contains precisely one rational curve.  However, a pseudo-groupless proper variety is not  covered by rational curves, i.e., a pseudo-groupless proper variety is non-uniruled.   
 \end{remark}

 It is natural to expect that  pseudo-grouplessness persists over field extensions. The following proposition confirms this expectation.
 \begin{proposition}Let $k\subset L$ be an extension of algebraically closed fields of characteristic zero. Let $X$ be a variety over $k$ and let $\Delta\subset X$ be a closed subscheme. If $X$ is groupless modulo $\Delta$ over $k$, then $X_L$ is groupless modulo $\Delta_L$ over $L$.
 \end{proposition}
 \begin{proof}(We follow the proof of \cite[Lemma~2.2]{JKa}.) Let $G$ be a finite type connected group scheme over $L$, let $U\subset G$ be a dense open with $\mathrm{codim}(G\setminus U)\geq 2$, and let $f:G\to X_L$ be  a non-constant morphism.   Now, by  standard spreading out arguments, we may  choose a finite type affine $k$-scheme $S = \Spec A$ with $A\subset L$, a finite type group scheme $\mathcal{G}\to S$ with geometrically connected fibres, a morphism $F:\mathcal{G}\to X\times S$ of $S$-schemes, and an isomorphism $\mathcal{G}_L \cong G$ of group schemes over $L$  such that $F_L:\mathcal{G}_L \cong G\to  X_L$ coincides with $f$. Replacing $S$ by a dense open subscheme if necessary, we may also arrange that $F_s:\mathcal{G}_s\to X\times \{s\} \cong X$ is non-constant for every $s$ in $S(k)$. Now, since $X$ is groupless modulo $\Delta$, for every $s$ in $S(k)$, the morphism $F_s: \mathcal{G}_s \to X$ factors over $\Delta$. This implies that the image of $F$ is contained in $\Delta \times S$, so that  $f$ factors over $\Delta_L$, as required.
 \end{proof}

\subsection{Birational invariance}  It is clear that the notion of grouplessness is not a birational invariant. Indeed, if $C$ is a smooth projective connected curve of genus $2$ and $X$ is the blow-up of $C\times C$ in a point, then $X$ is not groupless, although $C\times C$ is  groupless. However, the ``non-grouplessness'' of the surface $X$ is concentrated in the exceptional locus of $X\to C\times C$. In this section we  will show  that pseudo-grouplessness is a property which is stable under proper birational modifications.

\begin{lemma}\label{lem:birinv1_psgr}
Let $f:X\dashrightarrow Y$ be a  dominant  birational   map of proper varieties over $k$. 
Let $U$ be a dense open set of $X$ with complement $\Delta$  such that    $X$ is groupless modulo $\Delta $, 
 the rational map $f$ is well defined on $U$, and $f|_U:U\to Y$ is an open immersion.
Then $Y$ is groupless modulo $Y\setminus f(U)$.
\end{lemma}

\begin{proof}  
For a variety $V$ over $k$, we let  $S(V)$ denote the set of scheme-theoretic points $x\in V$ whose residue field $\kappa(x)$ is transcendental over $k$ and such that there is a connected finite type group scheme $G$ over $k$ such that $\kappa(x)$ is a subextension of $k\subset K(G)$, where $K(G)$ is the function field of $G$.   
Since $X$ is groupless modulo $X\setminus U$, it follows that $S(U)=\emptyset$.   Therefore, we have that $S(Y)\cap U=S(f(U))=\emptyset$, so that  $Y$ is groupless modulo $Y\setminus f(U)$.
\end{proof}

 \begin{lemma} \label{lem:birational_invariance}
Let $X$ and $Y$ be proper integral varieties over $k$. Assume that $X$ is birational to $Y$. Then $X$ is pseudo-groupless over $k$ if and only if $Y$ is pseudo-groupless over $k$.
 \end{lemma}
 \begin{proof} This follows from Lemma \ref{lem:birinv1_psgr}. 
 \end{proof}

 \subsection{A rigidity criterion}
 To study the notion of pseudo-grouplessness, we will use a  rigidity criterion for morphisms of varieties. The precise statement we need is given by Proposition \ref{propfaithfullyflatfactor} below. This result is presumably well-known to experts; due to lack of reference we include a proof.

\begin{lemma}\label{lemfracfunfun}Let  $C\hookrightarrow D$ be a faithfully flat morphism between integral domains. Then the following statement hold.
\begin{enumerate}
\item The sequence  of abelian groups
$$0\to C\to D\to D\otimes_C D, $$ where the morphism $d:D\to D\otimes_C D$ is defined by $x\mapsto x\otimes_C1-1\otimes_Cx,$ is exact 
\item We have that $C=D\cap \mathrm{Frac}(C)$ in $ \mathrm{Frac}(D)$.  
\end{enumerate}
\end{lemma}

\begin{proof}
To prove $(1)$, since $D$ is faithfully flat over $C$, we only need to prove that the following complex is exact:
$$0\to C\otimes_CD=D\to D\otimes_CD\to D\otimes_CD\otimes_CD.$$

We note that there exists a natural morphism $h:D\otimes_CD\otimes_C D\to D\otimes_C D$ sending $x\otimes_Cy\otimes_C z\mapsto y\otimes_C xz$.
Then $d\otimes_C {\rm id} (\sum x_i\otimes_C y_i)=0$ implies that $$\sum x_i\otimes_C1\otimes_C y_i-\sum 1\otimes_C x_i\otimes_C y_i=0.$$
Applying $h$, we get $$0=h(\sum x_i\otimes_C1\otimes_C y_i-\sum 1\otimes_C x_i\otimes_C y_i)$$
$$=\sum 1\otimes_C x_iy_i-\sum x_i\otimes_Cy_i.$$
It follows that $\sum x_i\otimes_C y_i\in C\otimes_CD$ which concludes the proof of $(1)$.

We now prove that $C= D\cap {\rm Frac}(C)$. First, note that it is clear that $C\subseteq D\cap {\rm Frac}(C)$. Thus, to conclude the proof, 
for every $g\in  D\cap {\rm Frac}(C)$, we need to show that $g\in C.$
To do so, write $g=s/t$ where $s,t\in C$ and $t\neq 0.$ Let $d:D\to D\otimes_CD$ be the morphism considered above.
It suffices to show that $d(g)=0$ in $D\otimes_CD.$
Since $D$ is flat over $C$ and $D\hookrightarrow {\rm Frac}(D)$ is an inclusion,
we get an inclusion $D\otimes_CD\hookrightarrow D\otimes_C  {\rm Frac}(D).$
Now we only need to show that $d(g)=0$ in $D\otimes_C  {\rm Frac}(D).$
We have
$$d(g)=s/t\otimes_C 1-1\otimes_C s/t=s/t\otimes_C t/t-1\otimes_C s/t$$
$$=ts/t\otimes_C 1/t-s\otimes_C 1/t=s\otimes t^{-1}-s\otimes t^{-1}=0.$$
This concludes the proof.
 \end{proof}

 \begin{lemma}\label{lemlocfactor} Let $W_1=\Spec R, W_2=\Spec S$, and $W_3=\Spec T$ be   affine varieties over $k$. 
Let $\pi:W_2\to W_1$ a faithfully flat morphism. Let $f: W_2\to W_3$ be a morphism. 
Assume that $W_1$ and $W_2$ are irreducible and normal, and that there exists a Zariski dense subset $E\subseteq W_1(k)$ such that the restriction of  $f$ on the fiber $\pi^{-1}(x)$ is constant for every $x\in E$.  Then there exists a unique morphism $h:W_1\to W_3$ such that $f=h\circ \pi.$
\end{lemma}
\begin{proof} The map $\pi$ induces an inclusion $\pi^*:R\hookrightarrow S$ and $f$ induces a morphism $f^*:T\to S.$  We only need to show that $f^*(T)\subseteq \pi^*(R).$

We note that $\pi^*$ extends to an embedding $\pi^*: {\rm Frac}(R)\hookrightarrow {\rm Frac}(S).$ We claim that $f^*(T) \subseteq  {\rm Frac(\pi^*(R))}.$  
Pick $g\in f^*(T)$.  
Pick an irreducible closed subvariety $Y\subseteq W_2$, such that $\pi|_Y: Y\to W_1$ is generically finite. 
Then there exists a nonempty affine open set $W_1'=\Spec R_1\subseteq W_1$ such that  $\pi|_Y$ is finite over $W_1'.$ Set $Y_1:=(\pi|_Y)^{-1}(W_1').$  Write $Y_1=\Spec Q.$
Then $\pi|_{Y_1}^*: R_1\hookrightarrow Q$ is finite, say of degree $d$. Let $\mathrm{tr}:Q\to R_1$ be the trace of $\pi|_{Y_1}^*$, and note that this is well-defined by the normality of $W_2$ and $W_1$. 
Note that  $$\pi^*(d^{-1}\mathrm{tr}(g|_{Y_1}))|_{Y_1}=g|_{Y_1}.$$
Since $g$ and $\pi^*(d^{-1}\mathrm{tr}(g|_{Y_1}))$ are both constant on $\pi^{-1}(x)$ for all $x\in E\cap W_1'$, it follows that 
$$\pi^*(d^{-1}\mathrm{tr}(g|_{Y_1}))|_{\pi^{-1}(E)}=g|_{\pi^{-1}(E)}.$$
Since $\pi^{-1}(E)$ is dense in $W_2$, we get $g=\pi^*(d^{-1}\mathrm{tr}(g|_{Y_1})).$ Since 
$$d^{-1}\mathrm{tr}(g|_{Y_1})\in R_1\subseteq {\rm Frac}(R),$$
we get $g\in \pi^*({\rm Frac}(R))={\rm Frac(\pi^*(R))}.$ This proves the claim.

It now follows from the claim and Lemma \ref{lemfracfunfun}.$(2)$ that  $f^*(T)\subseteq \pi^*(R).$ This concludes the proof of the lemma.
\end{proof}

 \begin{proposition}\label{propfaithfullyflatfactor} Let $\pi:Y\to B$  be a surjective and flat morphism between normal irreducible varieties over $k$. Let $Z$ be a variety over $k$ and let $f:Y\to Z$ be morphism. Let $E\subseteq B(k)$ be a Zariski dense subset.
Assume that for all point $x\in E$, the restriction $f|_{\pi^{-1}(x)}$ is constant.  Then there exists a morphism $h:B\to Z$ such that $f=h\circ \pi.$
\end{proposition}
 
\begin{proof}
Let $U_i$ be a finite affine open covering of $Y$.  After shrinking $U_i$, we may assume that every $f(U_i)$ is contained in some affine open subset $Z_i$ of $Z$/
Since $\pi$ is a finitely presented surjective and flat morphism, the subsets $V_i:=\pi(U_i)$ are a finite open covering of $B.$
By fppf descent,  replacing $Y, B,$ and $Z$ by $U_i, V_i$, and $Z_i$ respectively, we may and do assume that $Y$ and $Z$ are affine.    Let $V_i$ be a finite affine open covering of $B$. As $Y$ is affine and $B$ is separated,  it follows from  \cite[Tag~01SG]{stacks-project}    that the morphism $Y\to B$ is affine.  In particular,  the open subsets $W_i:=\pi^{-1}(V_i)$ are affine.

Thus, it suffices to treat the case where   $B$, $Y$ and $Z$ are affine in which case the result follows from Lemma \ref{lemlocfactor}.
\end{proof}

\subsection{Testing pseudo-grouplessness}

The notion of being pseudo-groupless is tested on maps from ``big'' opens from \emph{all} algebraic groups. In this section, we show that one can content oneself to testing the notion of pseudo-grouplessness for proper varieties on maps from ``big'' opens from abelian varieties; see Corollary \ref{cor:groupless2} for a precise statement.

We will make use of the following properties of linear algebraic groups.

\begin{lemma}\label{lem:groupschemes1}
If $G$ is a finite type connected affine group scheme over $k$ of dimension at least two, then $G$ has a connected proper non-trivial subgroup $H\subset G$.
\end{lemma}
\begin{proof}
If $G$ is commutative, then $G$ is isomorphic to the group scheme $ \mathbb{G}_{m,k}^s \times \mathbb{G}_{a,k}^t$ for some non-negative integers $s$ and $t$ (with $s+t\geq 2$) in which case the statement of the lemma is clear.

Thus, to prove the lemma, we may and do assume that $G$ is a non-commutative  group scheme.  Since $k$ is algebraically closed of characteristic zero and $G$ is a positive-dimensional affine finite type group scheme over $k$, we may choose an element   
 $g\in G(k)$   of infinite order.
Let $T$ be the identity component of the Zariski closure of the group $\{g^n \ | \ n\in \mathbb{Z}\}$ generated by $g$. Then $T$ is a non-trivial connected affine commutative group scheme.  Note that, as $G$ is non-commutative, we have that $G\neq T$. This shows that $G$ has a proper connected subgroup, as required.
\end{proof}

\begin{lemma}\label{lemggensubgroup} If $G$ is  a finite type connected affine group scheme over $k$ with $\dim G\geq 2$, then there exists a positive integer $s$ and  proper connected closed subgroups $H_1,\ldots, H_s$ of $G$ such that $G$ equals the subgroup $\langle H_1,\ldots, H_s\rangle$ generated by $H_1, \ldots, H_s$. 
\end{lemma}

\begin{proof}  
If $G$ is commutative, then $G$ is isomorphic to the group scheme $ \mathbb{G}_{m,k}^d \times \mathbb{G}_{a,k}^t$ for some non-negative integers $d$ and $t$ (with $d+t\geq 2$) in which case the statement of the lemma is clear.
Thus, we may and do assume that $G$ is a non-commutative group scheme.

Since $\dim G >1$, it follows from Lemma \ref{lem:groupschemes1} that $G$ has a proper connected subgroup.
Let $H$ be a maximal proper connected subgroup of $G$, and let $N(H):=\{g\in G|\,\, g^{-1}Hg=H\}$ be its normalizer.  By the maximality of $H$ and the connectedness of $N(H)$, we have that the identity component  $N(H)^0$ is either $H$ or $G$.

If  $N(H)^0=H,$ we may choose  an element $g\in G\setminus N(H)$. Note that the  subgroup  $g^{-1}Hg$ is a   connected proper subgroup of $G$. 
Since $H$ is maximal, we see that $\langle H, g^{-1}Hg\rangle=G$, as required.

 Thus, to prove the lemma, we may and do assume that $N(H)^0=G$, so that $N(H) = G$.  In this case, $H$ is a  normal subgroup of $G$. Since $H$ is maximal, there is no proper irreducible closed subgroup of $G/H$ of positive dimension. 
Therefore, it follows from Lemma \ref{lem:groupschemes1} that the (positive-dimensional) affine connected finite type group scheme $G/H$ is one-dimensional.   Since $k$ is of characteristic zero, there is an element $g_1\in G/H$ of infinite order.
Let $g\in G$ be an element with $\pi_H(g)=g_1$. Let $T$ be the identity component of the Zariski closure of $\{ g^n \ | \ n \in \mathbb{Z}\},$ and note that  $\pi(T)=G/H.$ It follows that $G=\langle H, T\rangle.$
Since $T$ is a commutative group  scheme and $G$ is non-commutative, we see that $T\neq G$. This concludes the proof.  
\end{proof}

\begin{lemma}\label{lemcaseglinear} Let $X$ be an integral variety over $k$ and let $\Delta$ be a   closed subscheme of $X$ such that, every non-constant morphism $\mathbb{G}_{m,k}\to X$ factors over $\Delta$. Let $G$ be a   finite type connected affine group scheme over $k$ and let $U$ be a dense open subset of $G$ with $\mathrm{codim}_G(G\setminus U)\geq 2$. 
If  $f:U\to X$ is a morphism with $f(U)\not\subseteq \Delta$,  then $f$ is constant.
\end{lemma}

\begin{proof} We may and do assume that $\Delta\neq X$.
We argue by induction on the dimension of $G$.
When $\dim G\leq 1$, the statement follows from the fact that $G$ contains a dense open isomorphic to $\mathbb{A}^1_k \setminus \{0\}$. Now assume that $\dim G\geq 2.$

To show that $f:U\to X$ is constant, it suffices to show that $f:G\dashrightarrow X$ is constant as a rational map.
Denote by $Z$ the Zariski closure of $f(U)$ in $X.$ Since $G$ (hence $U$) is irreducible, the closed subscheme $Z$ of $X$ is   irreducible.  To show that $f$ is constant, we may and  do assume that $Z=X$, so that $f:U\to X$ is dominant. (Indeed, note that $Z$ is a variety such that every non-constant morphism $\mathbb{G}_{m,k}\to Z$ factors over $\Delta\cap Z$. Moreover, the morphism $U\to Z$ is constant if and only if the morphism $f:U\to X$ is constant. Thus, we can replace $X$ by $Z$ to show that $f:U\to X$ is constant.) Since  the   morphism $f:U\to X$ is dominant, it induces an embedding of function fields
 $f^*(k(X))\subseteq k(G).$   The rational map $f:G\dashrightarrow X$ is constant if and only if $f^*(k(X))$ is contained in the subfield of constants $k\subseteq k(G).$

For every proper irreducible closed subgroup $H$ of $G$, denote by $\pi_H:G\to G/H$ the projection to the coset space. 
It induces an inclusion $\pi^*_H(k(G/H))\subseteq k(G).$ Observe that $\pi^*_H(k(G/H))=k(G)^H$.   (In this proof, we will only use the ``obvious'' inclusion  $\pi^*_H(k(G/H))\subseteq k(G)^H$.)
The map $\pi$ is finitely presented surjective and flat. It follows that $\pi_H(U)$ is a big open subset of $G/H.$ Moreover, there exists a dense open subscheme $V\subseteq \pi_H(U)$ such that, for every point $x$ in $V(k)$, the open subscheme $U\cap \pi_H^{-1}(x)$ of $\pi_H^{-1}(x)$  is big and satisfies   $f(\pi_H^{-1}(x))\not \subset \Delta.$ As a variety, $\pi_H^{-1}(x)$ is isomorphic to $H$, so that  the induction hypothesis implies that $f|_{\pi_H^{-1}(x)}$ is constant, for every $x$  in $V$.  It follows from Proposition \ref{propfaithfullyflatfactor}   that there exists $h_H:\pi_H(U)\to X$ such that $f=h_h\circ \pi_H.$
In particular, for every proper irreducible closed subgroup $H$ of $G$, we have 
$$f^*(k(X))\subseteq \pi^*_H(k(G/H))\subset k(G)^H. $$
  By Lemma \ref{lemggensubgroup}, there  exists a positive integer $s$ and  proper connected closed subgroups $H_1, \ldots, H_s$ of $G$ such that $G =\langle H_1,\ldots, H_s\rangle$. 
Therefore, we have that $$f^*(k(Z))\subseteq \cap_{i=1}^s k(G)^{H_i}=k(G)^{\langle H_1,\dots,H_s\rangle}=k(G)^G=k.$$
This concludes the proof. \end{proof}

 \begin{lemma}\label{lem:groupless} 
Let $X$ be an integral variety over $k$ and   let $\Delta$ be a closed subscheme of $X$. Suppose that    every  non-constant morphism $\mathbb{G}_{m,k}\to X$ factors over $\Delta$ and that, for every abelian variety $A$ over $k$ and every dense open subscheme $U\subset A$ with $\mathrm{codim}(A\setminus U)\geq 2$, every non-constant morphism $U\to X$ factors over $\Delta$. Then $X$ is groupless modulo $\Delta$ over $k$.
\end{lemma}

\begin{proof}   
Let $G$ be a finite type connected group scheme over $k$,  let $U\subset G$ be a dense open with $\mathrm{codim}(G\setminus U) \geq 2$, and let $f:U\to X$ be a morphism such that $f(U)\not\subseteq \Delta.$  To prove the lemma, we   need to show that $f$ is constant.

Let $H\subset G$ be the (unique)   normal affine connected subgroup of $G$ such that $A:=G/H$ is an abelian variety over $k$ (see \cite[Theorem~1.1]{ConradChevalley}). Denote by $\pi: G\to A$ the quotient map. Since $\pi$ is a finitely presented flat morphism, the image $V:=\pi(U)$ is open in $A.$ Moreover, the codimension of $A\setminus V$ in $A$ is at least two.  
For every $x\in A$, denote by $G_x:=\pi^{-1}(x)$ and $U_x:=U\cap G_x$.

Moreover, since $G\setminus U$ has codimension at least $2$ in $G$, there exists a dense open subset $V_1\subseteq V$ such that for every $x$ in $V_1$, the closed subset $G_x\setminus U$ of $G_x$ has codimension at least two.

Let $\eta$ be the generic point of $A$.  If $f(U_{\eta})\subseteq \Delta$, as $U_{\eta}$ is Zariski dense in $U$,  we have that  $f(U)\subseteq \Delta.$ This implies that $f(U_{\eta})\not\subseteq \Delta.$ Then $U_1:=(f|_U)^{-1}(X\setminus \Delta)$ is a nonempty open subset of $U.$ Since $\pi$ is flat, $V_2:=\pi(\pi^{-1}(V_1)\cap U_1)$ is open in $V_1.$

Note that, for every point $x\in V_2(k)$, the open subset  $U_x$ of $G_x$ is big and $f(U_x)\not\subseteq \Delta$. Therefore, by our assumption and   Lemma  \ref{lemcaseglinear}, we have that   $f|_{U_x}$ is constant for every $x$ in $V_2(k)$.
Since $V_2$ is Zariski dense in $A$, it follows from  Proposition \ref{propfaithfullyflatfactor} that
 there is a morphism $h:V\to X$ such that $f|_U=h\circ \pi|_U.$ Since $f(U)\not\subseteq \Delta$, it follows that $ h(V)\not\subseteq \Delta.$ Since $V$ is a big open subset of $A$, our assumption shows that $h$ is constant on $V$. It follows that $f$ is constant, as required.
\end{proof}

\begin{corollary}\label{cor:groupless2}
If $X$ is a proper integral variety $X$ over $k$ and $\Delta$ is a closed subscheme of $X$, then $X$ is groupless modulo $\Delta$ over $k$ if and only if, for every abelian variety $A$ over $k$ and every open subscheme $U\subset A$ with $\mathrm{codim}(A\setminus U)\geq 2$, every non-constant morphism of varieties $U\to X$ factors over $\Delta$.  
\end{corollary}
\begin{proof}   Suppose that, for every abelian variety $A$ over $k$ and every dense open subscheme $U\subset A$ with $\mathrm{codim}(A\setminus U)\geq 2$, every non-constant morphism of varieties $U\to X$ factors over $\Delta$.  
 Let $f:\mathbb{G}_{m,k}\to X$ be a  non-constant morphism. Since $X$ is proper over $k$, this morphism extends to a morphism $\overline{f}:\mathbb{P}^1_k\to X$. Since rational curves are dominated by elliptic curves, our assumption implies that the image of $\overline{f}$ (and thus $f$) is contained in $\Delta$.
It now follows from Lemma \ref{lem:groupless} that $X$ is groupless modulo $\Delta$.
\end{proof}

 \subsection{Applying the theorem of Hwang-Kebekus-Peternell}\label{section:hkp}
 If $X$ is a proper variety over $k$, we let $\mathrm{Aut}_{X/k}$ be the locally finite type group scheme over $k$ parametrizing automorphisms of $X$.  Its identity component $\mathrm{Aut}^0_{X/k}$ is a finite type connected group scheme over $k$.

 \begin{lemma}\label{lem:aut0}  
 Let $X$ be a projective integral variety over $k$ and let $\Delta\subsetneq X$ be a proper closed subset  such that, for every abelian variety $A$ over $k$, every non-constant morphism $A\to X$ factors over $\Delta$. Then $\mathrm{Aut}_{X/k}$ is zero-dimensional. 
 \end{lemma}
 \begin{proof} Since $\mathbb{P}^1_k$ is dominated by an elliptic curve, our assumption implies that $X$ is non-uniruled. In particular, it follows that $G:=\Aut^0_{X/k}$ is an abelian variety.   Let $x\in X\setminus \Delta$. Since $x$ lies in the image of the   morphism $G\to X$ defined by $g\mapsto g\cdot x$, the latter does     not factor over $\Delta$. Then, as $G$ is an abelian variety, our assumption implies that   the morphism $G\to X$ defined by $g\mapsto g\cdot x$ is constant. In particular, for every $g$ in $G$, the fixed locus $X^g$ of $g$ contains $X\setminus \Delta$. Since $X^g$ is closed, we conclude that $X^g = X$, i.e., the action of $g$ on $X$ is trivial. As the action of $G$ on $X$ is faithful, this implies that $G$ is trivial, as required.
 \end{proof}
 
If $X$ and $Y$ are projective varieties over $k$, we let $\underline{\mathrm{Hom}}_k(Y,X)$ be the scheme parametrizing morphisms $X\to Y$. Note that $\underline{\Hom}_k(Y,X)$ is a countable disjoint union of finite type schemes.  Moreover, we let $\underline{\mathrm{Sur}}_k(Y,X)$ be the scheme parametrizing surjective morphisms $Y\to X$.

 \begin{theorem} \label{thm:hkp0}
 Let $X$ be a projective normal integral variety over $k$ and let $\Delta$ be a closed subscheme  such that, for every abelian variety $A$ over $k$, every non-constant morphism $A\to X$ factors over $\Delta$.  Let $Y$ be a projective normal integral variety over $k$. Then the scheme $\underline{\mathrm{Sur}}_k(Y,X)$ is a countable disjoint union of zero-dimensional smooth projective schemes over $k$. 
 \end{theorem}
\begin{proof}
(We follow the proof of Theorem 2.8 in \cite{JKa}.)
Let $f:Y\to X$ be a surjective morphism. Let $\underline{\Hom}_k^f(Y,X)$ be the connected component of $\underline{\Hom}_k(Y,X)$ containing $f$.
By  \cite[Theorem~1.2]{HKP}, as $X$ is not covered by  rational curves, the scheme $\underline{\Hom}_k^f(Y,X)$ is smooth over $k$ and there exists a factorization $Y\to Z\to X$ and a surjective morphism $\Aut^0_{Z/k} \to \underline{\Hom}_k^f(Y,X)$ such that   $f:Z\to X$ is a finite surjective morphism and $\Aut^0_{Z/k}$ is an abelian variety over $k$.  Now, by our assumption on $X$, as  $Z\to X$ is finite,  for every abelian variety $A$ over $k$, it follows that   every non-constant morphism $A\to Z$ factors over the proper closed subset $f^{-1}(\Delta)$ of $Z$. Therefore, by Lemma \ref{lem:aut0}, the group scheme  $\Aut_{Z/k}^0$ is trivial. This implies that $\underline{\Hom}_k^f(Y,X)$ is isomorphic to $\Spec k$.  
\end{proof}

In \cite[Theorem~2.9]{JKa} it is shown that $\underline{\mathrm{Sur}}_k(Y,X)$ is zero-dimensional, assuming $X$ is groupless.  Our following result generalizes this statement for groupless varieties to pseudo-groupless varieties.
\begin{corollary}\label{cor:hkp}
If $X$ is  a pseudo-groupless projective normal integral variety over $k$ and $Y$ is a   projective normal integral variety over $k$, then $\underline{\mathrm{Sur}}_k(Y,X)$ is a countable disjoint union of zero-dimensional smooth projective schemes over $k$. 
\end{corollary}
\begin{proof}
Since $X$ is pseudo-groupless over $k$,  there is a proper closed subset $\Delta\subsetneq X$    such that, for every abelian variety $A$ over $k$, every non-constant morphism $A\to X$ factors over $\Delta$. Therefore, the result follows from Theorem \ref{thm:hkp0}.
\end{proof}

\begin{remark}
There are projective varieties $X$ which are \textbf{not} pseudo-groupless over $k$, but for which the conclusion of the theorem above still holds, i.e.,  for every projective integral variety $Y$ over $k$, the scheme $\underline{\mathrm{Sur}}_k(Y,X)$ is smooth over $k$ \textbf{and} zero-dimensional. For example, let   $X$ be the  blow-up of a simple abelian surface $A$  in its identity element.  Then, for every abelian variety $B$ over $k$, every morphism $B\to X$ is constant.  
\end{remark}

\begin{remark}
Let $X$ be a (smooth projective) K3 surface over $k$. Then $\underline{\mathrm{Sur}}(Y,X)$ is zero-dimensional. More generally, if $X$ is a  simply connected  smooth projective connected variety over $k$ with $\mathrm{h}^0(X,T_X)=0$, then $\underline{\mathrm{Sur}}_k(Y,X)$ is zero-dimensional; see \cite{HKP}.
\end{remark}

\subsection{Classifying surfaces of general type} 
If $X$ and $Y$ are projective varieties over $k$, then we let $\mathrm{Dom}_k(Y,X)$ be the set of dominant rational maps $Y\dashrightarrow X$. The finiteness (or infinitude) of $\mathrm{Dom}_k(Y,X)$ is a property of the birational equivalence class of $X$. That is,  if $X'$ is a projective variety over $k$ which is birational to $X$ over $k$, then $\mathrm{Dom}_k(Y,X) = \mathrm{Dom}_k(Y,X')$.

 \begin{lemma}\label{lem:groupless_surfaces}
Let $X$ be a proper pseudo-groupless integral two-dimensional variety over $k$. Then   $X$ is of general type and, for every projective integral variety $Y$, the set  $\mathrm{Dom}_k(Y,X)$ is finite. 
 \end{lemma}
 \begin{proof} (If $X$ is groupless, then the lemma follows essentially from \cite[Corollary~1.6.(2)]{ShangZhang}, see also \cite{JAut}.)  To prove the lemma,  let $X_{min}$ be the minimal model of $X$ (see \cite[Chapter~9]{Liu2}). Since pseudo-grouplessness is a birational invariant for proper varieties (Lemma \ref{lem:birational_invariance}),  we may and do assume that $X$ is a  minimal smooth projective surface. Moreover,    if $X$ is of general type, then $\mathrm{Dom}_k(Y,X)$ is finite for every projective integral variety $Y$; see\cite{Iitaka}. Therefore, it suffices to show that $X$ is of general type.
 
 Since $X$ is pseudo-groupless, we see that $X$ is not uniruled. Therefore,  the Kodaira dimension of $X$ is  $0, 1$, or $2$. Since $X$ is pseudo-groupless, the variety $X$ is not covered by elliptic curves. In particular, as smooth proper surfaces of Kodaira dimension $1$ admit an elliptic fibration (over a curve), we see that the Kodaira dimension of $X$ is $0$ or $2$. To conclude the proof, it suffices to show that the Kodaira dimension of $X$ is not $0$.

Since $X$ is a projective pseudo-groupless variety over $k$, for every abelian  surface $A$, every morphism $A \to X$ is non-surjective.   In particular, the surface  $X$ does not admit a finite \'etale cover $Y\to X$ with $Y$ an abelian surface.  
 Thus,  if $X$ is of Kodaira dimension zero, by the classification of surfaces, it is a K3 surface up to a finite \'etale   cover. However, a K3 surface    is covered by elliptic curves  
\cite[Corollary~2.2]{HuybrechtsBook}, and is  thus not pseudo-groupless.  We conclude  that $X$ is not covered by a K3 surface, so that $X$ is of Kodaira dimension two as required.
 \end{proof}

Conjecturally, a projective variety of general type is pseudo-groupless. We show that the finiteness of the set $\mathrm{Dom}_k(Y,X)$ for every $Y$ is equivalent to being of general type when $\dim X <3$.  We start with the following property of K3 surfaces.

 \begin{proposition}\label{prop:k3}
 Let $X$ be a K3 surface over $k$.  Then there is a projective integral variety $Y$ over $k$ such that $\mathrm{Dom}_k(Y,X)$ is infinite.
 \end{proposition}
 \begin{proof}
 Choose $Y\to X$ dominant and $Y\to B$ a family of elliptic curves (with a section). Such data exists by \cite[Corollary~13.2.2]{HuybrechtsBook}. Since $Y$ has infinitely many   endomorphisms, the set $\mathrm{Dom}_k(Y,X)$ is infinite. This concludes the proof.
 \end{proof}
 
 \begin{corollary}\label{cor:surfaces}
Let $X$ be a projective integral variety of dimension at most two over $k$ such that, for every projective integral variety $Y$ over $k$, the set $\mathrm{Dom}_k(Y,X)$ is finite. Then  $X$ is of general type.
 \end{corollary}
\begin{proof} (We follow essentially the proof of  \cite[Theorem~II]{GreenGriffithsTwoApps}). If $X'$ is a projective variety which is birational to $X$, then $\mathrm{Dom}_k(Y,X')$ is finite for every projective integral variety $Y$ over $k$. Therefore, to prove the corollary,  we may and do assume that $X$ is a minimal  smooth projective   surface over $k$. Let $\kappa$ be the Kodaira dimension of $X$. If $\kappa =1$, then $X$ admits an elliptic fibration $X\to B$ from which it is not hard to see that  there is a finite cover $B'\to B$ such that the pull-back $X':=X\times_B B'$ has a section and (therefore) admits   infinitely many dominant rational self-maps    (induced by the group structure of the generic fibre of $X'\to B'$). Therefore, $\kappa\neq 1$. If $\kappa =0$, then $X$ is dominated by a K3 surface or an abelian surface.  By Proposition \ref{prop:k3} and the fact that every abelian surface has infinitely many surjective endomorphisms, we see that $\kappa \neq 0$.  Suppose that $\kappa =-\infty$. Then, by the classification of such surfaces, there is a smooth projective connected curve $C$ over $k$ such that $X$ is birational to $\mathbb{P}^1_k\times_k C$. However, since the latter surface has infinitely many surjective endomorphisms, we see that $\kappa \neq -\infty$.  We conclude that $\kappa = 2$, as required.
\end{proof}

\begin{theorem} Let $X$ be a proper integral variety of dimension at most two over $k$. Then $X$ is of general type if and only if, for every projective variety $Y$ over $k$, the set $\mathrm{Dom}_k(Y,X)$ is finite.
\end{theorem}
 
\begin{proof}
If $X$ is of general type, then it follows from Kobayashi-Ochiai's theorem that, for every projective integral variety $Y$ over $k$, the set $\mathrm{Dom}_k(Y,X)$ is finite (see \cite{Iitaka}). Conversely,  if $X$ is a projective variety of dimension at most two over $k$ such that, for every projective integral variety $Y$ over $k$, the set $\mathrm{Dom}_k(Y,X)$ is finite, then $X$ is of general type. This is Corollary \ref{cor:surfaces}. 
\end{proof}

\section{Pseudo-algebraic hyperbolicity and pseudo-boundedness}\label{section:alg}
In this section we study the boundedness of maps from curves into  varieties. We first ``pseudify'' Demailly's notion of algebraic hyperbolicity following \cite{vBJK} (see also \cite{JKa}).

 \begin{definition}[Algebraic hyperbolicity modulo a subset] 
 Let $X$ be a projective variety over $k$ and let $\Delta$ be a closed subscheme of $X$. We say that $X$ is \emph{algebraically hyperbolic over $k$ modulo $\Delta$} if, for every ample line bundle $L$ on $X$, there is a real number $\alpha(X,\Delta,L)$ depending only on $X$, $\Delta$ and $L$ such that, for every smooth projective connected curve $C$ over $k$ and every morphism $f:C\to X$ with $f(C)\not \subset \Delta$, the inequality
 \[
 \deg_C f^\ast L \leq \alpha(X,\Delta,L) \cdot \mathrm{genus}(C)
 \] holds.
 \end{definition}

\begin{definition}
 A projective variety  $X$ over $k$ is \emph{pseudo-algebraically hyperbolic (over $k$)} if there is a proper closed subset $\Delta\subsetneq X$ such that $X$ is algebraically hyperbolic modulo $\Delta$.
\end{definition}

 \begin{definition}[Boundedness modulo a subset] Let $n\geq 1$ be an integer, 
 let $X$ be a projective variety over $k$, and let $\Delta$ be a closed subscheme of $X$. We say that $X$ is \emph{$n$-bounded   over $k$ modulo $\Delta$} if, for every normal projective variety $Y$ of dimension at most $n$, the scheme $\underline{\Hom}_k(Y,X)\setminus \underline{\Hom}_k(Y,\Delta)$ is of finite type over $k$. We say that $X$ is \emph{bounded over $k$ modulo $\Delta$} if, for every $n\geq 1$, the variety $X$ is $n$-bounded modulo $\Delta$.
 \end{definition}
 
\begin{definition} Let $n\geq 1$ be an integer.
 A projective variety  $X$ over $k$ is \emph{pseudo-$n$-bounded over $k$} if there is a proper closed subset $\Delta\subsetneq X$ such that $X$ is $n$-bounded  modulo $\Delta$.   
\end{definition}

 \begin{definition}
 A projective variety is \emph{pseudo-bounded over $k$} if, for every integer $n\geq 1$, we have that $X$ is pseudo-$n$-bounded over $k$.
 \end{definition}

 In the definition of algebraic hyperbolicity one asks for the existence of a uniform and linear bound in terms of the genus. If we instead   ``only'' ask for a uniform bound in the genus, we  retrieve  an a priori weaker notion of hyperbolicity which is referred to as \emph{weak boundedness} by Kov\'acs-Lieblich \cite{KovacsLieblich}. Their notion can be shown to coincide with  the notion of \emph{boundedness} as defined above; see  \cite[\S 9]{vBJK}.

 \begin{lemma}\label{lem:birrie1}
 Let $f:X\to Y$ be a proper birational surjective morphism of projective integral varieties over $k$. Let $\Delta_X\subset X$ be a closed subscheme. If $X$ is $1$-bounded modulo $\Delta_X$ (resp. algebraically hyperbolic modulo $\Delta_X$) over $k$, then $Y$ is $1$-bounded modulo $f(\Delta_X)\cup f(\mathrm{Exc}(f))$ (resp. algebraically hyperbolic modulo $f(\Delta_X)\cup f(\mathrm{Exc}(f))$) over $k$.
 \end{lemma}
 \begin{proof} Define $\Delta_Y := f(\Delta_X)\cup f(\mathrm{Exc}(f))$.
 Let $L_Y$ be an ample line bundle and write $f^\ast L_Y = L_X - A$, where $L_X$ and $A$ are ample line bundles on $X$. Assume that $X$ is $1$-bounded modulo $\Delta_X$, and let $\alpha(X,L_X,\Delta_X)$ be the real number in the definition of $1$-boundedness modulo $\Delta_X$. 
 
 If $g:C\to Y$ is a morphism with $g(C)\not\subset \Delta_Y$,  then we can lift $g:C\to Y$ to a morphism $\overline{g}:C\to X$. We see that
 \[
 \deg_C g^\ast L_Y = \deg_C \overline{g}^\ast f^\ast L_Y = \deg_C \overline{g}^\ast(L_X - A) \leq \deg_C \overline{g}^\ast L_X.
  \]
This implies that $\deg_C g^\ast L_Y$ is bounded by the real number  $\alpha(X,L_X,\Delta_X)$, so that $Y$ is $1$-bounded modulo $\Delta_Y$, as required.

A minor modification of the above arguments shows that, if $X$ is  algebraically hyperbolic modulo $\Delta_X$ over $k$, then $Y$ is   algebraically hyperbolic modulo $f(\Delta_X)\cup f(\mathrm{Exc}(f))$ over $k$. This concludes the proof.
 \end{proof}
 
 \begin{proposition}[Birational invariance]\label{prop:bir_inv_boundedness}
 Let $X$ and $Y$ be projective integral varieties over $k$. Assume that   $X$ and $Y$ are birationally equivalent over $k$. Then the following statements hold.
 \begin{enumerate}
 \item The projective variety $X$ is pseudo-algebraically hyperbolic over $k$ if and only if $Y$ is pseudo-algebraically hyperbolic over $k$.
 \item The projective variety $X$ is pseudo-$1$-bounded over $k$ if and only if $Y$ is pseudo-$1$-bounded over $k$.
 \end{enumerate}  
 \end{proposition}
 \begin{proof}  We may and do assume that there is a proper birational surjective morphism $f:X\to Y$.  
 
  Let us first assume that $Y$ is pseudo-algebraically hyperbolic over $k$. Thus, let  $\Delta_Y\subsetneq Y$ be a proper closed subset of $Y$ such that $Y$ is algebraically hyperbolic     modulo $\Delta_Y$, and let $\alpha(Y,L_Y, \Delta_Y)$ be the real number in the definition of algebraic hyperbolicity modulo $\Delta_Y$. Let $L_Y$ be an ample line bundle on $Y$. Let $n\geq 1$ be an integer and let $H$ be an effective divisor on $X$ with support in  $\mathrm{Exc}(f)$  such that $L_X:=f^*L_Y^{\otimes n} - H$ is ample on $X$; such a divisor exists by the negativity lemma \cite[Lemma~3.39]{KollarMori}. Note that $n$ only depends on $f:X\to Y$ and $L_Y$. Let $\Delta_X = f^{-1}(\Delta_Y)\cup \mathrm{Exc}(f)$. Then, for every smooth projective connected curve $C$ over $k$   and every  morphism $g:C\to X$ with $g(C)\not\subset \Delta_X$, we have that 
 \[
 \deg_C g^\ast L_X = n\deg_C (fg)^\ast L_Y - \deg_C g^\ast H \leq n  \deg_C (fg)^\ast L_Y \leq n \cdot \alpha(Y,L_Y, \Delta_Y) \cdot \mathrm{genus}(C).
 \] Indeed, the first equality is straightforward, and the first inequality uses that the intersection of the effective  divisor $H$ with the image of $C$ in $X$ is non-negative, as $H$ is supported on $\mathrm{Exc}(f)$ and $f(C)\not\subset \mathrm{Exc}(f)$.  The last inequality follows from the fact that $Y$ is algebraically hyperbolic modulo $\Delta_Y$ and implies  that $X$ is pseudo-algebraically hyperbolic.
 
 If $Y$ is pseudo-$1$-bounded over $k$, then a similar line of reasoning shows that $X$ is pseudo-$1$-bounded over $k$.  By Lemma \ref{lem:birrie1} this concludes the proof.
 \end{proof}

 \begin{proposition}\label{prop:trivs}
 Let $X$ be a projective variety over $k$ and let $\Delta$ be a closed subscheme. 
Then, the following statements hold.
\begin{enumerate}
\item      If $X$ is algebraically hyperbolic modulo $\Delta$ over $k$, then $X$ is bounded modulo $\Delta$ over $k$.
\item  Let $k\subset L$ be an extension of algebraically closed fields. If $X$ is algebraically hyperbolic modulo $\Delta$ (resp. bounded modulo $\Delta$) over $k$, then $X_L$ is algebraically hyperbolic modulo $\Delta_L$  (resp. bounded modulo $\Delta_L$) over $L$.
\end{enumerate}
\end{proposition}
\begin{proof}
See \cite[\S 9]{vBJK}.
\end{proof}
\begin{proposition}\label{prop:1b_bir_b}  Assume that $k$ is \textbf{uncountable}. Let $X$ be a projective integral variety over $k$, and let   $Y$ be a projective integral variety  which is birational to $X$  over $k$. Then,  $X$ is pseudo-$1$-bounded over $k$ if and only if $Y$ is pseudo-bounded over $k$.
\end{proposition}
 \begin{proof}
 If $X$ is pseudo-$1$-bounded over $k$, it follows from  Proposition \ref{prop:bir_inv_boundedness} that $Y$ is pseudo-$1$-bounded over $k$. It then follows from \cite[Lemma~9.10]{vBJK} that   $Y$ is pseudo-bounded over $k$. This concludes the proof.
 \end{proof}

\section{Pseudo-geometric hyperbolicity}\label{section:geomhyp}
It is natural to study the finiteness (as opposed to boundedness) of pointed curves on a hyperbolic variety.  Following \cite{vBJK, JBook} we formalize this property and refer to it as being ``geometrically hyperbolic''.

 \begin{definition}[Geometric hyperbolicity modulo a subset]   Let $X$ be a  variety over $k$ and let $\Delta$ be a closed subscheme of $X$. We say that $X$ is \emph{geometrically hyperbolic   over $k$ modulo $\Delta$} if, for every $x$ in $X(k)\setminus \Delta$, every smooth quasi-projective connected curve $C$  over $k$ and every $c$ in $C(k)$,  the set $\Hom_k((C,c),(X,x))$ of morphisms $f:C\to X$ with $f(c) = x$ is finite.   
 \end{definition}

\begin{definition} Let $n\geq 1$ be an integer.
 A variety $X$ over $k$ is \emph{pseudo-geometrically hyperbolic over $k$} if there is a proper closed subset $\Delta\subsetneq X$ such that $X$ is geometrically hyperbolic   modulo $\Delta$. 
\end{definition}

  A variety $X$  over $k$ is  {geometrically hyperbolic over $k$} if it is geometrically hyperbolic modulo the empty subset over $k$, i.e., for every smooth quasi-projective connected curve $C$ over $k$, every $c$ in $C(k)$, and every $x$ in $X(k)$, the set $\Hom_k((C,c),(X,x))$ of morphisms of $k$-schemes $f:C\to X$ with $f(c)=x$ is finite. Note that geometric hyperbolicity is not a birational invariant. However, if $X$ and $Y$ are proper integral varieties over $k$ which are birational over $k$, then $X$ is pseudo-geometrically hyperbolic over $k$ if and only if $Y$ is pseudo-geometrically hyperbolic over $k$.

  \begin{lemma}\label{lem:geomhyp_persists} Let $k\subset L$ be an extension of \textbf{uncountable} algebraically closed fields of characteristic zero,   let $X$ be a variety over $k$, and let $\Delta$ be a closed subscheme of $X$.
  If  $X$ is  geometrically hyperbolic modulo $\Delta$ over $k$, then $X_L$ is geometrically hyperbolic modulo $\Delta_L$ over $L$.
  \end{lemma}
  \begin{proof}  This follows from \cite[Lemma~2.4]{JLitt} when $\Delta =\emptyset$. As is stated in \cite[Remark~9.15]{vBJK}, the statement of the lemma can be proven by adapting the arguments of the proof of \cite[Corollary~9.7]{vBJK}. 
  \end{proof}

\begin{proposition}\label{prop:geo_is_grp}
Let $X$ be a proper variety over $k$ and let $\Delta\subsetneq X$ be a proper closed subset of $X$. If $X$ is geometrically hyperbolic modulo $\Delta$ over $k$, then $X$ is groupless modulo $\Delta$ over $k$.
\end{proposition}
\begin{proof} 
Suppose that $X$ is not groupless modulo $\Delta$ over $k$. Assume that  there exists a non constant morphism 
$f:U\to X$ whose image is not contained in $\Delta$, where $U$ is a big open subset of an abelian variety $A$ over $k$.
Denote by $Y$ the Zariski closure of $f(U)$, and note that $Y\cap \Delta \neq Y$.
Choose a point $0\in U$ with $f(0)\not\in \Delta$ such  that $f:U\to Y$ is smooth at $0$ (so that $f(0)$ is nonsingular in $Y$). We may and do assume that $0$ is the identity element of $A$.
Let $K\subseteq T_0A$ be  the kernel of $df$ at $0$.
Now, let $C$ be a smooth projective irreducible curve  in $A$ which contains $0$ and such that  $T_0 C \not\subset K$.   For $n\geq 1$, let $[n]:A\to A$ be the multiplication by $n$ on $A$.
Then, for every $n\geq 1$, as $0\in [n](C)$, the rational map $f\circ [n]|_C: C \dashrightarrow  Y\subseteq X$ is well-defined and   non-constant.
Moreover, since $X$ is proper over $k$ and $C$ is a smooth curve over $k$, for every $n\geq 1$, the rational map  $f\circ [n]|_C$ extends uniquely to a morphism $f_n:= f\circ [n]|_C: C \to  Y\subseteq X$.
Note that the morphism $ f_n$ is smooth at $0$ by construction and its differential $df_n: T_0C\to T_{f(0)}Y$ equals  $n\cdot  {d}f$. As $ {d}f \neq 0$, these differentials are pairwise distinct. Therefore, 
since the morphisms $f_1, f_2, \ldots$ have pairwise distinct differential, they are pairwise distinct.  Since   $$f_1, f_2, \ldots \in \Hom_k((C,0),(X,f(0)))\setminus \Hom_k((C,0),(\Delta,f(0)))$$ is an infinite sequence of pairwise distinct morphisms, we conclude that $X$ is not geometrically hyperbolic modulo $\Delta$.   
\end{proof}

\begin{corollary} Let $X$ be a projective normal integral variety over $k$, and let  $Y$ be a projective normal integral variety over $k$.
If $X$ is pseudo-geometrically hyperbolic over $k$, then $\underline{\mathrm{Sur}}_{k}(Y,X)$ is   zero-dimensional, and smooth over $k$.
\end{corollary}
\begin{proof}
Since  $X$ is pseudo-groupless over $k$ (Proposition \ref{prop:geo_is_grp}), the result follows from Corollary \ref{cor:hkp}.
\end{proof}

   \begin{corollary}  \label{cor:1b_is_gr}
 Let $X$ be a projective variety over $k$ and let $\Delta\subset X$ be a closed subscheme.   If $X$ is $1$-bounded modulo $\Delta$, then $X$ is groupless modulo $\Delta$.
\end{corollary}
 \begin{proof} Since $X$ is $1$-bounded modulo $\Delta$, it follows that $X$ is geometrically hyperbolic modulo $\Delta$; see \cite[Proposition~9.17]{vBJK}. Therefore,   it follows from   Proposition \ref{prop:geo_is_grp} that $X$ is groupless  modulo $\Delta$.
 \end{proof}

\begin{corollary}\label{cor:1b_is_rigid} Let $X$ be a projective normal integral variety  over $k$, and let $Y$ be  a projective normal integral variety over $k$.
If $X$ is pseudo-$1$-bounded over $k$, then $\underline{\mathrm{Sur}}_{k}(Y,X)$ is   zero-dimensional, and smooth over $k$.
\end{corollary}
\begin{proof}
Since  $X$ is pseudo-groupless over $k$ (Corollary  \ref{cor:1b_is_gr}), the result follows from Corollary \ref{cor:hkp}.
\end{proof}

  \section{Finiteness results for pseudo-bounded varieties}\label{section:fin}
  
  In this section we prove finiteness results for pseudo-bounded varieties, and especially for  the set $\mathrm{Sur}_k(Y,X)$ of surjective morphisms $Y\to X$ with $X$ pseudo-bounded. We start with the following proposition which says that, roughly speaking, given a projective pseudo-bounded variety $X$ over $k$ and a general point $x$ in $X(k)$, the set of pointed maps $(Y,y)\to (X,x)$  is finite.
  
    \begin{proposition}\label{prop:1b_is_geomhyp2} Let $X$ be a projective variety over $k$ and let $\Delta\subsetneq X$ be a proper closed subset of $X$ such that  $X$ is bounded modulo $\Delta$ over $k$. Then, for every  integral projective variety $Y$ over $k$, every $y$ in $Y(k)$ and $x$ in $X\setminus \Delta$, the set $\Hom_k((Y,y),(X,x))$ is finite. 
  \end{proposition}
  \begin{proof} Let $Y$, $y$, and $x$ be as in the statement, and suppose that $\Hom_k((Y,y),(X,x))$ is infinite with $Y$, $y$ and $x$. Let  $$f_1, f_2, \ldots \in \Hom_k((Y,y),(X,x))$$ be pairwise distinct elements, and note  that $f_i\not\in \Hom_k(Y,\Delta)$. Since the scheme $$\underline{\Hom}_k((Y,y),(X,x))\setminus \underline{\Hom}_k(Y,\Delta)$$ is of finite type over $k$,  the possible Hilbert polynomials of all the $f_i$ form a finite set  (depending only on $X, \Delta,$ and $C$). Then, by  Mori's bend-and-break \cite[Proposition~3.1]{Debarrebook1}, there is a rational curve in $X$ containing $x$ contradicting the fact that every rational curve in $X$ is contained in $\Delta$. (To see that $X$ has a rational curve containing $x$, let $Y'\to Y$ be a smooth projective desingularization of $Y$, and  let $L/k$ be an uncountable algebraically closed field extension of $k$. Note that the pairwise distinct morphisms $f_{i,L}:Y'_L\to X_L$ have bounded degree. For every pair $(i,j)$ of positive integers, let $\Gamma_{i,j}\subset Y'(L)$ be the closed set of points $P$ in $Y'(L)$ such that $f_i(P) = f_j(P)$.  Since $L$ is uncountable, we have that $\cup_{i\neq j} \Gamma_{i,j} \neq Y'(L)$. Let $P\in Y'(L)\setminus \cup_{i\neq j}\Gamma_{i,j}$. By Bertini's theorem, there is a smooth projective connected curve  $C\subset Y'_L$  containing $c:=y$ and the point $P$. In particular, the morphisms $f_i|_C: C\to X_L$ have bounded degree (because the morphisms $f_i:Y\to X$ have bounded Hilbert polynomial), satisfy $f_i(c) = x$, and are pairwise distinct (as $f_i(P)\neq f_j(P)$ for all $i\neq j$).  Thus, by \cite[Proposition~3.1]{Debarrebook1}, there is a rational curve $\mathbb{P}^1_L\to X_L$ whose image contains $x$. By a standard spreading-out and specialization argument, it follows that $X$ has a rational curve $\mathbb{P}^1_k\to X$ whose image contains $x$.)
  \end{proof}
  
  If $X$ and $Y$ are varieties over $k$, we let $\mathrm{Sur}_k(Y,X)$ be the set of surjective morphisms $Y\to X$ over $k$. If $X$ and $Y$ are projective over $k$, then $\mathrm{Sur}_k(Y,X)$ is the set of $k$-points of the scheme $\underline{\mathrm{Sur}}_k(Y,X)$ (see Section \ref{section:hkp}).

  \begin{lemma}\label{lem:fin_1}
Let $n$ be a positive integer and let $X$ be a pseudo-$n$-bounded projective integral variety    over $k$. Then, for every   projective integral variety $Y$ over $k$ of dimension at most $n$, the set $\mathrm{Sur}_k(Y,X)$ is finite.
\end{lemma}
\begin{proof} Replacing $Y$ by    its normalization if necessary, we may and do assume that $Y$ is normal.  Since every surjective morphism $Y\to X$ factors uniquely over the normalization of $X$, we may and do assume that $X$ is normal.
 Let $\Delta\subsetneq X$ be a proper closed subset of $X$ such that $X$ is $n$-bounded modulo $\Delta$.  From the fact that $X$ is  $n$-bounded modulo $\Delta$, it follows readily that $X$ is $1$-bounded modulo $\Delta$. Indeed, as $X$ is $n$-bounded modulo $\Delta$, for every normal projective variety $Y$ of dimension at most $n$, the scheme $\underline{\Hom}_k(Y,X) \setminus \underline{\Hom}_k(Y,\Delta)$ is of finite type over $k$. In particular, as $n>0$, this applies to one-dimensional normal projective varieties $Y$ over $k$.
 
Note that $\mathrm{Sur}_k(Y,X)  = \mathrm{Sur}_k(Y,X) \setminus \Hom_k(Y,\Delta)$. Thus, by $n$-boundedness of $X$ modulo $\Delta$, it follows that $\underline{\mathrm{Sur}}_k(Y,X)$ is of finite type.   Now, since  $X$ is a projective pseudo-$1$-bounded      normal integral variety over $k$, by Corollary \ref{cor:1b_is_rigid}, the scheme $\underline{\mathrm{Sur}}_k(Y,X)$ is   zero-dimensional. Since $\underline{\mathrm{Sur}}_k(Y,X)$ is zero-dimensional and of finite type over $k$, we conclude that $\mathrm{Sur}_k(Y,X)$ is finite, as required.   
\end{proof}
  
    \begin{theorem} \label{thm:surr} Let $X$ be a pseudo-bounded projective integral variety over $k$. If $Y$ is a projective integral variety over $k$, then the set $\mathrm{Sur}_k(Y,X)$ is finite.
  \end{theorem}
 \begin{proof}
Note that, for every integer $n\geq 1$, the projective variety $X$ is pseudo-$n$-bounded (by definition), so that the result    follows from Lemma \ref{lem:fin_1}.
 \end{proof}

\begin{corollary}\label{cor:b_is_all} 
 Let $X$ be a projective variety over $k$ and let $\Delta\subsetneq X$ be a proper closed subset of $X$ such that   is bounded    modulo $\Delta$ over $k$. Then, for $A\subset X$ a non-empty integral closed subscheme  not contained in $\Delta$, and $B\subset X$ a non-empty integral closed subscheme, the set of morphisms $f:Y\to X$ with $f(B) = A$ is finite. 
\end{corollary}
 \begin{proof}
 Since $A$ is not contained in $\Delta$, we have that $A\cap \Delta\neq A$. Therefore, as $A$ is bounded modulo $A\cap \Delta$, it follows that $A$ is pseudo-bounded, so that the set of surjective morphisms $B\to A$ is finite (Theorem \ref{thm:surr}). Now, fix a  surjective morphism $\psi:B\to A$. Then, it suffices to show  that the set $\Hom_k ((Y,B),(X,A);\psi)$ of morphisms $f:Y\to X$ with $f(B) = A$ and $f|_B = \psi$ is finite. To do so, let $y\in B(k)$ and let $x= \psi(b)$. Then, the   set $\Hom_k((Y,B),(X,A);\psi)$ is contained  in  $\Hom_k((Y,y),(X,x))$. The latter set is finite by Proposition \ref{prop:1b_is_geomhyp2}. This concludes the proof.
 \end{proof}

We now prove Theorem \ref{thm:sur_intro1} as stated in the introduction.
 \begin{proof}[Proof of Theorem \ref{thm:sur_intro1}]
Since $X$ is pseudo-$1$-bounded over $k$ and $k$ is uncountable, it follows that $X$ is pseudo-bounded over $k$ (Proposition \ref{prop:1b_bir_b}). The finiteness of $\mathrm{Sur}(Y,X)$ now follows from Lemma \ref{lem:fin_1}.
 \end{proof}
 
 \begin{remark}[Noguchi's theorem]\label{remark:noguchi}
 Since Brody hyperbolic projective varieties over $\mathbb{C}$ are $1$-bounded  by Arzela-Ascoli's theorem and Brody's Lemma, we see that Theorem \ref{thm:sur_intro1} implies that, for $X$ a Brody hyperbolic projective variety over $\mathbb{C}$ and $Y$ a projective integral variety over $\mathbb{C}$, the set $\mathrm{Sur}_{\mathbb{C}}(Y,X)$ is finite. The latter statement was a conjecture of Lang which was   proven by Noguchi \cite{Noguchi13}; see \cite[p.~303]{KobayashiBook} for a discussion of the history of Lang's conjecture. Note that Theorem \ref{thm:sur_intro1} is more general than Noguchi's result, as it allows for an ``exceptional'' locus along which the variety is \textbf{not} hyperbolic.
\end{remark} 
 
 We now prove Theorem \ref{thm:final}.  
 
 \begin{proof}[Proof of Theorem \ref{thm:final}]Since $X$ is algebraically hyperbolic modulo $\Delta$,    the projective variety $X$ is bounded modulo $\Delta$  over $k$ (see Proposition \ref{prop:trivs}). Therefore, the second statement follows from Corollary \ref{cor:b_is_all}. The first statement follows from the definition of boundedness modulo $\Delta$.
 \end{proof}
 
We now note that Theorem \ref{thm:sur_intro} is an immediate consequence of Theorem \ref{thm:final}.
 
\begin{proof}[Proof of Theorem \ref{thm:sur_intro}]
This follows from Theorem \ref{thm:final} by taking $B= Y$ and $A=X$.
\end{proof}

 We also have   the following finiteness result for $1$-bounded pairs $(X,\Delta)$ over an uncountable field.

\begin{corollary}  Assume that $k$ is \textbf{uncountable}.
 Let $X$ be a projective variety over $k$ and let $\Delta\subsetneq X$ be a proper closed subset of $X$ such that   is $1$-bounded    modulo $\Delta$ over $k$. Then, for $A\subset X$ a non-empty integral closed subscheme  not contained in $\Delta$, and $B\subset X$ a non-empty integral closed subscheme, the set of morphisms $f:Y\to X$ with $f(B) = A$ is finite. 
\end{corollary}
  
 \begin{proof}
 Since $k$ is uncountable and $X$ is $1$-bounded modulo $\Delta$, it follows from Proposition \ref{prop:1b_bir_b}  that $X$ is bounded modulo $\Delta$. Therefore, the corollary is an immediate consequence of Corollary \ref{cor:b_is_all}. 
 \end{proof}

\section{Pseudo-Mordellic varieties}\label{section:mordell} For proper varieties, we defined the notion of pseudo-Mordellicity in the introduction following Lang's original definition \cite{Lang2}. 
In this section we define the notion of Mordellicity modulo a closed subset for possible non-proper varieties; this definition also appears in Vojta's article \cite{VojtaLangExc}.

  The main result of this section is that pseudo-Mordellic proper varieties are pseudo-groupless (Theorem \ref{thm:ar_is_gr}). As an application we deduce the rigidity of surjective morphisms from a fixed variety to a fixed pseudo-Mordellic projective variety (Corollary \ref{cor:sur_is_rigid}).

\begin{definition}
Let $X\to S$ be a morphism of schemes. Assume that $S$ is integral with function field $K(S)$. Then, we define $X(S)^{(1)}$ to be the subset of elements $P$ in $X(K(S))$ such that, for every point $s$ in $S$ of codimension one, the point $P$ lies in the image of $X(\mathcal{O}_{S,s}) \to X(K(S))$.
\end{definition}

In \cite{VojtaLangExc} Vojta refers to the elements of $X(S)^{(1)}$ as \emph{near-$S$-integral points} of $X$. Note that, if $S$ is a one-dimensional noetherian normal integral scheme, then $X(S)^{(1)} = X(S)$.

\begin{definition}
 Let $X$ be a variety over $k$ and let $\Delta$ be a closed subset of $X$. We say that $X$ is \emph{Mordellic   modulo $\Delta$} if, for every regular $\ZZ$-finitely generated subring $A\subset k$ of $k$ and every model $\mathcal{X}$ for $X$ over $A$, we have that every positive-dimensional irreducible component of the Zariski closure of $\mathcal{X}(A)^{(1)}$ in $X$ is contained in $\Delta$.
 \end{definition}

 \begin{remark}\label{remark:proper} Let $X$ be a proper variety over $k$ and let $\Delta$ be a closed subset of $X$. 
Then, by the valuative criterion of properness, the variety  $X$  is Mordellic modulo $\Delta$ if and only if, for every finitely generated subfield $K\subset k$ and every proper model $\mathcal{X}$ over $K$, the set $\mathcal{X}(K)\setminus \Delta$ is finite.  Thus, the above definition extends the definition of pseudo-Mordellicity we gave in the introduction for proper varieties.
 \end{remark}
 
 \begin{definition}
 A variety $X$ over $k$ is \emph{pseudo-Mordellic   over $k$} if there is a proper closed subset $\Delta\subsetneq X$ such that $X$ is Mordellic    modulo $\Delta$. Moreover, $X$ is \emph{Mordellic over $k$}   if $X$ is Mordellic modulo the empty subset.  
 \end{definition}

 \begin{remark}[Examples]
 A one-dimensional variety $X$ over $k$ is   Mordellic over $k$ if and only if it is  groupless over $k$. This follows from Faltings's theorem  \cite{Faltings2, FaltingsComplements}. We refer to   \cite{Autissier1, Autissier2, BauerStoll,  CLZ, Dimitrov, FaltingsLang, Levin, Moriwaki,  VojtaSubb, Vojta1a, Vojta2b, Vojta1} for more examples of Mordellic   varieties.
Faltings proved that a closed subvariety $X$ of an abelian variety $A$ is Mordellic modulo its special locus (or Ueno locus) $\mathrm{Sp}(X)$; see \cite{FaltingsLang}.
Moriwaki used Faltings's theorem to prove that 
 a smooth projective connected variety  $X$ over $k$ whose cotangent bundle $\Omega^1_X$  is ample and globally generated  is  Mordellic  \cite{Moriwaki} (see also  \cite{ATdV}). 
 \end{remark}

 The birational invariance of the notion of pseudo-Mordellicity is essentially built into its definition. Indeed,  the infinitude of the set of near-integral points is preserved under proper birational modifications. 
 
To prove the rigidity of surjective morphisms to a proper pseudo-Mordellic variety, we will use that such varieties are pseudo-groupless. In fact, we prove the following more precise statement using the potential density of rational points on abelian varieties; this was proven by Frey-Jarden \cite{FreyJarden} (and later also by Hassett--Tschinkel \cite[\S3]{HassettTschinkel}).

 \begin{theorem}\label{thm:ar_is_gr}   Let $X$ be a proper variety over $k$ and let $\Delta\subset X$ be a closed subscheme.
 If $X$ is  Mordellic modulo $\Delta$ over $k$ , then $X$ is  groupless  modulo $\Delta$ over $k$.
 \end{theorem}
 \begin{proof}  By Corollary \ref{cor:groupless2}, it suffices to show that, for every abelian variety $A$ over $k$, every  non-constant rational map $A\dashrightarrow X$ factors over $\Delta$. Let $A\dashrightarrow X$ be a non-constant rational map, and let $B\to A$ be a proper birational surjective morphism such that the composed rational map $B\to A\dashrightarrow X$ is a  (non-constant) morphism. It suffices to show that this morphism $f:B\to X$ is constant. To do so, we use \cite[\S3]{JAut} and choose a finitely generated subfield $K\subset k$, a model $\mathcal{B}$ for $B$ over $K$, a model $\mathcal{A}$ for $A$ over $K$ with $\mathcal{A}(K)$ Zariski-dense in $A$, and a  proper birational surjective morphism $\mathcal{B}\to \mathcal{A}$ defining $B\to A$ over $K$. Since $\mathcal{B}\to \mathcal{A}$ is proper birational surjective and $\mathcal{A}(K)$ is dense in $A$, it follows that $\mathcal{B}(K)$ is dense. Replacing $K$ by a finitely generated field extension if necessary, we may and do assume that there is a model $\mathcal{X}$ for $X$ over $K$ and a morphism $F:\mathcal{B}\to \mathcal{X}$ with $F_k=f$. Now, let $Z:=\mathrm{Im}(F)$ be the image of $F$, and note that $Z$ is a  positive-dimensional irreducible closed subscheme of $\mathcal{X}$. Since  $Z(K)$ is dense in $Z$ and $\dim Z >0$, it follows from the properness of $X$ over $k$ and Remark \ref{remark:proper} that $Z$ is contained in $\Delta$.
 This concludes the proof.
 \end{proof}

 \begin{corollary}\label{cor:sur_is_rigid}
 Let $X$ be a projective  pseudo-Mordellic normal integral  variety over $k$. Then, for every normal integral projective variety $Y$ over $k$, the scheme $\underline{\mathrm{Sur}}_k(Y,X)$ is zero-dimensional and smooth over $k$.
 \end{corollary}
 \begin{proof}
 Since a proper pseudo-Mordellic variety over $k$ is pseudo-groupless over $k$ (Theorem \ref{thm:ar_is_gr}), the corollary follows from the rigidity of surjective morphisms $Y\to X$ to a pseudo-groupless normal projective variety $X$; see Corollary \ref{cor:hkp}. 
 \end{proof}

\begin{remark}
Let $X$ be a proper surface over $k$. If $X$ is pseudo-Mordellic over $k$, then  $X$ is of general type.
To prove this, note that $X$ is pseudo-groupless (Theorem \ref{thm:ar_is_gr}), so that the claim follows from the fact that pseudo-groupless proper surfaces are of general type   (Lemma  \ref{lem:groupless_surfaces}). In particular, a K3 surface over $k$ is  not pseudo-Mordellic over $k$.
\end{remark}

\subsection{Arithmetic hyperbolicity} The reader might have noticed that we refer to varieties with finitely many near-integral points as being Mordellic, whereas we also employ the term ``arithmetically hyperbolic'' in some other papers \cite{vBJK, JLevin, JSZ, JAut, JLitt, JLalg}. In fact, there is a subtle difference between these notions which we will explain briefly in this section; see   \cite[\S7]{JBook} for a more detailed discussion. 
 
 Let $X$ be a variety over $k$ and let $\Delta$ be a closed subscheme of $X$. We say that $X$ is \emph{arithmetically hyperbolic over $k$ modulo $\Delta$} if, for every $\ZZ$-finitely generated subring $A\subset k$ of $k$ and every model $\mathcal{X}$ for $X$ over $A$, we have that every positive-dimensional irreducible component of the Zariski closure of $\mathcal{X}(A)$ in $X$ is contained in $\Delta$.
 With this definition at hand, a variety $X$ over $\Qbar$ is arithmetically hyperbolic over $\Qbar$ (as defined in \cite{JAut} and \cite[\S 4]{JLalg}) if and only if $X$ is arithmetically hyperbolic over $\Qbar$ modulo the empty subset.
 
As in the case of grouplessness and Mordellicity, we will say that  a variety $X$ over $k$ is \emph{pseudo-arithmetically hyperbolic over $k$} if there is a proper closed subset $\Delta\subsetneq X$ such that $X$ is arithmetically hyperbolic modulo $\Delta$.
 
   It is   clear from the definitions that, if $X$ is Mordellic modulo $\Delta$ over $k$, then $X$ is arithmetically hyperbolic modulo $\Delta$ over $k$. In particular, a pseudo-Mordellic variety is pseudo-arithmetically hyperbolic and a Mordellic variety is arithmetically hyperbolic. Indeed,  roughly speaking, to say that a variety is arithmetically hyperbolic 
  is to say that any set of integral points on it is finite, and to say that a variety is Mordellic is to say that any set of ``near-integral'' points on it is finite. The latter sets are a priori bigger.  However, there is no difference between these two sets when $k=\Qbar$. That is, a variety $X$ over $\Qbar$ is arithmetically hyperbolic modulo $\Delta$ if and only if it is Mordellic modulo $\Delta$ over $\Qbar$. Thus, the \emph{a priori} difference between  Mordellicity and arithmetic hyperbolicity can only occur over algebraically closed fields of positive transcendence degree over $\QQ$. 
  
In this paper we work with pseudo-Mordellic   varieties  instead of pseudo-arithmetically hyperbolic varieties because, if $X$ and $Y$ are proper integral varieties over $k$ which are birational over $k$, then $X$ is pseudo-Mordellic over $k$ if and only if $Y$ is pseudo-Mordellic over $k$.
  This property is crucial in our proof of Theorem \ref{thm:birs}, and we do not know whether the analogous statement holds for pseudo-arithmetic hyperbolicity.

  \section{Dominant self-maps of pseudo-Mordellic   varieties}\label{section:ends}
 
  Let $k$ be an algebraically closed field of characteristic zero. 
 If $f:X\dashrightarrow X$ is a dominant rational self-map of a variety over $k$ and $i\geq 1$, we let $f^i:X\to X$ be the composition of $f$ with itself $i$-times, and we let $\mathrm{In}(f)$ be the locus of indeterminacy of $f$.  
  
In \cite{Amerik2011}  Amerik proved that a dominant rational self-map of a variety    which is  not of finite order has points of infinite order. The precise statement reads as follows.

 \begin{theorem} [Amerik]\label{thm:Amerik} Let $k$ be  an algebraically closed field of characteristic zero, let $X$ be a variety over $k$, and let $f:X\dashrightarrow X$ be a  dominant rational self-map of infinite order. Then, for every non-empty open subset $U\subset X$, there is a point $x$ in $U\setminus \mathrm{In}(f)$ such that the orbit of $x$ is infinite and is contained in  $U\setminus \mathrm{In}(f)$.
\end{theorem}

If $k$ is uncountable, then one can prove Amerik's theorem using standard ``Noetherian induction'' type arguments. However, Amerik's theorem  requires non-trivial arguments when the base field is countable. 
Indeed, 
the arguments in Amerik's paper  rely on the work of many authors on algebraic dynamics; see for instance  
  \cite{BellBook, Ben2012, Fakhruddin, GhiTuc2009, Poonenqp, Silverman}. In addition, Amerik uses Hrushovski's theorem on intersections of graphs with Frobenius \cite{HHH, Varshavsky}.
  
 Amerik proved her theorem when $k=\Qbar$ in \cite{Amerik2011}, but  Amerik's theorem holds  (as stated above)  with $k$ an arbitrary algebraically closed field of characteristic zero. This is explained   in \cite[\S 6]{JAut}, and \cite[Theorem~5.6]{Xie0}. In fact, this version of Amerik's theorem has  been used throughout the literature; see for instance \cite[Lemma~3.3]{BellGhioca2} and \cite{Xie1}.

 The main result of this paper relies on the following application of Amerik's theorem  to proper pseudo-Mordellic varieties.

 \begin{theorem}\label{thm:amerik2}    If $X$ is a   proper pseudo-Mordellic integral variety over $k$ and $f:X \dashrightarrow X $ is  a dominant rational self-map, then $f$ is a birational self-map of finite order.
 \end{theorem}
 \begin{proof}  We argue by contradiction. Suppose that $f$ has infinite order. Let $\Delta\subsetneq X$ be a proper closed subset of $X$ such that $X$ is Mordellic modulo $\Delta$.  Let $U:=X\setminus \Delta$. By Amerik's theorem (Theorem \ref{thm:Amerik}), there is a point $x$ in $U\setminus \mathrm{In}(f)$ such that the orbit of $x$ is infinite and contained in $U\setminus \mathrm{In}(f)$. Now, to arrive at a contradiction, let $K\subset k$ be a finitely generated subfield, let $\mathcal{X}$ be a proper model for $X$  over $K$ such that $x$ lies in the subset $\mathcal{X}(K) $ of $X(k)$, let $\mathcal{D}\subset \mathcal{X}$ be a model for $\Delta\subset X$ over $K$, and let $F:\mathcal{X}\dashrightarrow \mathcal{X}$ be a dominant rational self-map such that $F_k  = f$. Then, the infinite orbit of $x$ is contained in $\mathcal{X}(K)\setminus D$. However, as $X$ is Mordellic modulo $\Delta$, the set $\mathcal{X}(K)\setminus D = \mathcal{X}(K)\setminus \Delta$ is finite. This concludes the proof. 
 \end{proof}
  
  \begin{remark}    
 Given a pseudo-Mordellic proper variety $X$ over $k$ and an uncountable algebraically closed field $L$ containing $k$, it is not clear whether the variety $X_L$ is pseudo-Mordellic over $L$, although it certainly should be by the Persistence Conjecture \cite[Conjecture~1.20]{vBJK}; this expectation is verified    in certain situations in  \cite{vBJK,   JLevin, JSZ, JAut,  JLitt}. This is the reason we need the full force of Amerik's theorem over countable fields to prove Theorem \ref{thm:amerik2}.
  \end{remark}

 \begin{proof}[Proof of Theorem \ref{thm:birs}] Let $X$ be a proper pseudo-Mordellic variety over $k$. Let $f:X\dashrightarrow X$ be a dominant rational self-map of $X$ over $k$, and note  that $f$ is a birational self-map of finite order (Theorem \ref{thm:amerik2}).  Thus, to conclude the proof, it suffices to show that $\mathrm{Bir}_k(X)$ is finite. However, as $\mathrm{Bir}_k(X)$ is torsion and $X$ is non-uniruled,   we conclude that $\mathrm{Bir}_k(X)$ is finite (Theorem \ref{thm:tor}). 
 \end{proof}

 \section{Dominant self-maps of pseudo-$1$-bounded varieties} In this section we prove the finiteness of the set of dominant rational self-maps $X\dashrightarrow X$ for $X$ a pseudo-$1$-bounded projective variety  (Theorem \ref{thm:birs2}).
 We start with using Amerik's theorem to prove that the automorphism group of a pseudo-$1$-bounded projective variety is torsion.  (This follows from the results in Section \ref{section:fin} when the base field is uncountable.)

 \begin{lemma}\label{lem:new}
 If $X$ is a pseudo-$1$-bounded projective variety over $k$, then $\Aut_k(X)$ is torsion.
 \end{lemma}
 \begin{proof}  
Let $\Delta\subsetneq X$ be a proper closed subset of $X$ such that $X$ is $1$-bounded modulo $\Delta$ over $k$.  Let $f\in \Aut_k(X)$ and note that, by Amerik's theorem (Theorem \ref{thm:Amerik}),    there is a point $x$ in $X(k)  $ such that the forward $f$-orbit  $\{x, f(x), f^2(x),\ldots\}$ of $x$ is     contained in $X(k)\setminus \Delta$. Let $L$ be an ample line bundle on $X$, and let $C$ be a smooth complete intersection of divisors in $\vert L\vert $ containing $x$. Then, as  $X$ is $1$-bounded modulo $\Delta$ over $k$, there is a real number $\alpha$ depending only on $C$, $X$, $\Delta$, and $L$, such that, for every $m\geq 1$, the inequality
 \[\deg_C (f^{m})^\ast L \leq \alpha\]
 holds. This implies that the iterates of $f$ are contained in a finite type subgroup scheme of $\mathrm{Aut}_{X/k}$. In particular, there is an integer $m$ such that $f^m$ is contained in $\mathrm{Aut}_{X/k}^0$. However, as $X$ is  pseudo-$1$-bounded, it follows from its pseudo-grouplessness (Corollary \ref{cor:1b_is_gr}) that $\mathrm{Aut}^0_{X/k}$ is trivial (see Corollary \ref{cor:hkp}). This shows that $f^m$ is trivial, and thus that $f$ has finite order. This concludes the proof of the lemma.
 \end{proof}
 
As we will see below,  the fact that the automorphism group of $X$ is torsion can be leveraged to show that the birational automorphism group of $X$ is torsion. To do so, we will use the following well-known theorem of Weil.
 
 \begin{theorem}[Weil's Regularization Theorem]\label{thm:weil}
 Let $X$ be a projective integral variety over $k$ and let $f:X\dashrightarrow X$ be a birational self-map. Let $L$ be an ample line bundle on $X$ and let $C\subset X$ be a smooth one-dimensional connected complete intersection  of   divisors in $|L|$ and  let   $\alpha \in \RR$  be such that, for all $m\geq 1$, the inequality  
  $$\deg_C (f^m)^{\ast} L \leq \alpha$$ holds. Then, there is a projective integral variety $Y$ birational to $X$ such that the birational self-map $f:Y\dashrightarrow Y $ defined by $f$ extends to an automorphism $F:Y\to Y$. 
 \end{theorem}
 \begin{proof} This follows from
\cite[Theorem~2.5]{Cantat2}.  
 \end{proof}

 In the proof of the following   result we will   use Amerik's theorem,  Weil's regularization theorem, and the theory of dynamical degrees. It is the ``geometric'' analogue of Theorem \ref{thm:amerik2}.

 \begin{proposition}\label{prop:bir_is_finite}    Let $X$ be a projective integral pseudo-$1$-bounded variety over $k$. Then,   every   dominant rational self-map $X\dashrightarrow X$ is   birational, and $\mathrm{Bir}_k(X)$ is a  torsion group.
 \end{proposition}
 \begin{proof}  Since pseudo-$1$-boundedness is a birational invariant (Proposition \ref{prop:bir_inv_boundedness}), by resolving the singularities of $X$, we may and do assume that $X$ is smooth over $k$.  
 
 Let $\Delta\subsetneq X$ be a proper closed subset of $X$ such that $X$ is $1$-bounded modulo $\Delta$ over $k$.
Let $f:X\dashrightarrow X$ be  a dominant rational self-map of $X$ over $k$.
 
By Amerik's theorem (Theorem \ref{thm:Amerik}), there is a point $x$ in $X(k)$ such that the forward $f$-orbit $\{x,f(x),f^2(x),\ldots\}$  is well-defined and contained in $X(k)\setminus \Delta$. Let $L$ be an ample line bundle on $X$, and let  $C$ be a smooth one-dimensional connected complete intersection of divisors in $\vert L\vert$ containing $x$.
 Since  $f^m(x) \not\in \Delta$ for every $m\geq 0$ by construction,    the morphisms 
 $f^m|_C:C\to X$  define elements of the finite type $k$-scheme $\underline{\Hom}_k(C,X)\setminus \underline{\Hom}_k(C,\Delta)$. Therefore,   there is a real number $\alpha>0$ depending only on $C$, $L$, $\Delta$, and $X$ such that, for all $m\geq 0$, the inequality 
\begin{align}\label{ineq}
\deg_C (\left(f^{m}|_C\right)^{\ast} L) \leq \alpha
\end{align} holds.

We now show that $f$ is birational. Suppose for a contradiction that $f$ is not birational.  We now use properties of dynamical degrees to deduce a contradiction.  Let $d=\dim X$. Then, for  $i=0,\ldots,d$, the $i$-th dynamical degree of $f$ is defined by 
$$\lambda_i:=\lim_{n\to \infty}(L^{d-i}\cdot (  (f^{n})^\ast L)^i)^{1/n}, $$
where we write $L^m$ for the $m$-th self-intersection number of $L$ on $X$. The fact that this is a well-defined non-negative real number is proven in \cite{DinhSibony1}.
Note that $\lambda_0=1$ and that $\lambda_d$ is the (generic) degree   of $f$. 
Moreover, the $i$-th dynamical degree of $f$ is independent of the choice of $L$ and the sequence $(\lambda_0,\ldots, \lambda_d)$ is log-concave \cite[Theorem~1.1.(3)]{TTT} (see also \cite{Dan, DinhNguyen, DinhSibony2, OguisoDyn}), 
so that 
$$  \lambda_d\leq \lambda_1^d.$$
Since $\deg f>1$ (by assumption), we have that $\lambda_d>1$. In particular, we have $\lambda_1^d>1$, and thus $\lambda_1 >1$. It follows that $(L^{d-1}\cdot (f^m)^*L)$ tends to infinity as $m$ tends to infinity. By our definition of $C$, it follows that $\deg_C ((f^{m})^\ast L) $ tends to infinity as $m$ tends to infinity. This contradicts   inequality (\ref{ineq}) above. We conclude that $f$ is birational.

Since $f$ is birational, we may apply Weil's regularization theorem  (Theorem \ref{thm:weil}) to see that   there is a projective integral variety $Y$ over $k$ which is birational to $X$ and such that $f:X\dashrightarrow X$ defines an automorphism $F:Y\to Y$. However, as    $Y$ is birational to $X$, it follows that $Y$ is pseudo-$1$-bounded over $k$ (Proposition \ref{prop:bir_inv_boundedness}), so that $\Aut_k(Y)$ is torsion (Lemma \ref{lem:new}).  Therefore, $F$ has finite order. We conclude that $f$ has finite order, as required. 
 \end{proof}
 
 With Proposition \ref{prop:bir_is_finite} at hand, we are now ready to prove the finiteness of the set of dominant rational self-maps of a pseudo-$1$-bounded projective variety over $k$.
 \begin{proof}[Proof of Theorem \ref{thm:birs2}]
 Since $X$ is pseudo-$1$-bounded over $k$, it is clear that  $X$ is non-uniruled.  Moreover, by Proposition \ref{prop:bir_is_finite}, the group $\mathrm{Bir}_k(X)$ is torsion and every dominant rational self-map $X\dashrightarrow X$ is birational.  As $X$ is non-uniruled and $\Bir_k(X)$ is torsion, it follows from  Theorem \ref{thm:tor} that  $\mathrm{Bir}_k(X)$ is finite.   This concludes the proof.
 \end{proof}

 \begin{corollary}\label{cor:birs2}
 If $X$ is a pseudo-bounded (resp. pseudo-algebraically hyperbolic) projective integral variety over $k$, then the set of dominant rational self-maps $X\dashrightarrow X$ is finite.
 \end{corollary}
 \begin{proof}
As $X$ is obviously pseudo-$1$-bounded over $k$, this follows from   Theorem \ref{thm:birs2}.
 \end{proof}

 \bibliography{refs_new}{}
\bibliographystyle{plain}

\end{document}